\documentclass[11pt]{amsart}

\usepackage{lipsum}
\usepackage{amsfonts,bm,amssymb,amsmath}
\usepackage{graphicx}
\usepackage{epstopdf}
\usepackage[caption=false]{subfig}
\usepackage{pgfplots}
\usepackage{algorithmic}
\usepackage{amsrefs}

\usepackage{amsopn}
\usepackage{amsrefs}
\usepackage{float}
\usepackage{booktabs}

\usepackage[normalem]{ulem}
\usepackage{geometry}

\usepackage{hyperref}
\usepackage{cleveref}

\usepackage{color}
\usepackage{diagbox}

\ifpdf
\DeclareGraphicsExtensions{.eps,.pdf,.png,.jpg}
\else
\DeclareGraphicsExtensions{.eps}
\fi

\newtheorem{thm}{Theorem}[section]

\newtheorem{lemm}{Lemma}[section]
\newtheorem{rema}{Remark}[section]
\newtheorem{remark}{Remark}[section]

\hypersetup{
	colorlinks,
	linkcolor={red!50!black},
	citecolor={blue!50!black},
	urlcolor={blue!80!black}
}

\numberwithin{equation}{section}

\definecolor{newcolor1}{rgb}{.8,.349,.1}
\colorlet{bblue}{blue!50!black}
\crefformat{equation}{(#2#1#3)}
\crefmultiformat{equation}{(#2#1#3)}{ and~(#2#1#3)}{, (#2#1#3)}{ and~(#2#1#3)}

\crefformat{figure}{Figure~#2#1#3}
\crefmultiformat{figure}{Figures~ #2#1#3}{ and~#2#1#3}{, (#2#1#3)}{ and~(#2#1#3)}
\crefformat{table}{Table~#2#1#3}

\def\X{\mbox{\boldmath $X$}}

\def\e{\mbox{\boldmath $e$}}
\def\f{\mbox{\boldmath $f$}}
\def\F{\mbox{\boldmath $F$}}
\def\g{\mbox{\boldmath $g$}}

\def\m{\mbox{\boldmath $m$}}

\def\q{\mbox{\boldmath $q$}}
\def\x{\mbox{\boldmath $x$}}

\newcommand\dtk {{k}}
\def\0{\mbox{\boldmath $0$}}
\def\um{\underline{\bm m}}

\begin{document}

\title[Convergence analysis for the LL equation with large damping]{Convergence analysis of a second-order accurate, linear numerical scheme for the Landau-Lifshitz equation with large damping parameters}

\author{Yongyong Cai}
\address{Laboratory of Mathematics and Complex Systems and School of Mathematical Sciences \\ Beijing Normal University\\ Beijing\\ China.}
\email{yongyong.cai@bnu.edu.cn}

\author{Jingrun Chen}
\address{School of Mathematical Sciences\\ Soochow University\\ Suzhou\\ China.}
\email{jingrunchen@suda.edu.cn (Corresponding author)}

\author{Cheng Wang} 
\address{Mathematics Department\\ University of Massachusetts\\ North Dartmouth \\ MA 02747\\ USA.}
\email{cwang1@umassd.edu}

\author{Changjian Xie} 
\address{Department of Mathematics\\ The Hong Kong University of Science and Technology \\ Clear Water Bay\\ Kowloon\\ Hong Kong\\ China.}
\email{macjxie@ust.hk}


\date{\today}














\begin{abstract}
A second order accurate, linear numerical method is analyzed for the Landau-Lifshitz equation with large damping parameters. This equation describes the dynamics of magnetization, with a non-convexity constraint of unit length of the magnetization. 
The numerical method is based on the second-order backward differentiation formula in time, combined with an implicit treatment of the linear diffusion term and explicit extrapolation for the nonlinear terms. Afterward, a projection step is applied to normalize the numerical solution at a point-wise level. This numerical scheme has shown extensive advantages in the practical computations for the physical model with large damping parameters, which comes from the fact that only a linear system with constant coefficients (independent of both time and the updated magnetization) needs to be solved at each time step, and has greatly improved the numerical efficiency. Meanwhile, a theoretical analysis for this linear numerical scheme has not been available. In this paper, we provide a rigorous error estimate of the numerical scheme, in the discrete $\ell^{\infty}(0,T; \ell^2)  \cap \ell^2(0,T; H_h^1)$ norm, under suitable regularity assumptions and reasonable ratio between the time step-size and the spatial mesh-size. In particular, the projection operation is nonlinear, and a stability estimate for the projection step turns out to be highly challenging. Such a stability estimate is derived in details, which will play an essential role in the convergence analysis for the numerical scheme, if the damping parameter is greater than 3. 
\end{abstract}

%

\maketitle

\section{Introduction}
%


The Landau-Lifshitz (LL) equation \cite{Landau1935On}, with quasilinearity and the constraint of unit length of magnetization, describes the evolution of the magnetization in ferromagnetic materials with applications of information storage in the magnetic-based recording devices~\cite{Prohl2001Computational}. The nonlinear conservative term of the LL equation preserves the unit length of magnetization and drives the system. The remaining nonlinear part related to the harmonic mapping of the LL equation is dissipated by a factor of damping parameters.   
Such a parameter plays an important role for energy evolution, which can be calculated  \cite{Miranda2021MechanismsBL}. 

There have been extensive numerical works for the LL equation \cite{GilbertKelly1955, LiZengPRLGiant2019, Miranda2021MechanismsBL, Tanaka2014MicrowaveAssistedMR}. One of the most popular temporal discretization is the semi-implicit method \cite{jingrun2019analysis, cimrak2005error, gao2014optimal}, which turns out to remarkably relax restrictions of temporal step-size. 
For example, the linearly implicit backward Euler approach has been well studied in \cite{Lubich2021,alouges2006convergence,jingrun2019analysis,gao2014optimal}. A combination of this numerical idea with a high-order non-conforming finite element discretization in space has been  proposed and analyzed in \cite{alouges2006convergence}, in which a projection is applied to an approximate tangent space to the normality constraint. Moreover, a convergence analysis in both space and time has been established in \cite{Lubich2021}, by evaluating the approximated error of time derivative term which is orthogonal to the magnetization. The error estimates for linearly implicit schemes, based on either backward Euler or Crank-Nicolson method, combined with finite element/finite difference spatial discretization, have been obtained in \cite{An2021,An2016OptimalEE, gao2014optimal}. The backward differentiation formula (BDF)-based linearly implicit methods have been analyzed in~\cite{Lubich2021, jingrun2019analysis}, and a second-order accuracy have been rigorously proved under the same condition that temporal step-size proportional to the spatial mesh-size in both space and time. 

Meanwhile, it is noticed that, all these existing numerical methods lead to an asymmetric linear system of equations with the coefficients dependent of the updated magnetization. An efficient numerical solver for such an asymmetric linear system is highly non-trivial, which usually results in a very expensive computation cost in the three-dimensional simulation. Therefore, a numerical scheme only involved with a constant coefficient linear system, so that the coefficients are independent of the updated magnetization, is highly desirable. 
In fact, such a linear numerical scheme has been proposed and studied in a recent work~\cite{Xie2021BDF2Linear}. In more details, the second-order BDF stencil is used in the temporal discretization, the perfect Laplace term (in the harmonic mapping part) is treated implicitly, while the nonlinear terms are approximated by explicit extrapolation formulas. After an intermediate magnetization vector is obtained by this linear algorithm, a projection of magnetization onto the unit sphere is applied, to satisfy the non-convex constraint of unit length. Of course, this  numerical approach leads to a linear system with constant coefficients independent the updated magnetization at each time step. Because of this subtle fact, the linear numerical scheme has demonstrated great advantages in the simulation of ferromagnetic materials for large damping parameters \cite{Xie2021BDF2Linear} . Furthermore, extensive simulation experiments have indicated that the proposed linear numerical method preserves better stability property as the damping parameter takes large values, in comparison with all the existing works~\cite{Lubich2021, alouges2006convergence, An2016OptimalEE, Boscarino2016High, jingrun2019analysis}, etc.  

On the other hand, a theoretical analysis of the proposed linear numerical scheme has not been available, in spite of its extensive advantages in the practical computations for large damping parameters. The key theoretical difficulty is associated with the fact that, a fully explicit treatment of the nonlinear gyromagnetic term (by an extrapolation formula) breaks its (energetic) conservative feature at the numerical level, so that a direct control of this nonlinear term becomes a very challenging issue. The only hopeful approach is to control this term by the linear diffusion term in the harmonic mapping part, while the fact that the nonlinear gyromagnetic term and the linear diffusion term are updated by different temporal discretization makes this estimate highly non-trivial. In this article, we provide the convergence analysis and the optimal rate error estimate for the proposed linear numerical scheme, in the discrete $\ell^{\infty}(0,T; \ell^2)  \cap \ell^2(0,T; H_h^1)$ norm, if the damping parameter is greater than 3. To overcome the above-mentioned difficulties, we build the stability estimate of the projection step, which will play a crucial role in the rigorous error estimate for the original error function. In particular, a standard $\ell^2$ stability estimate for the projection step is not sufficient to recover the convergence analysis, and an $H_h^1$ stability estimate turns out to be necessary at the projection step, which comes from the technique of controlling the nonlinear gyromagnetic term by the linear diffusion term. In more details, the \textit{a-priori} $W_h^{1,\infty}$ estimate for the numerical solution and \textit{a-priori} $H_h^1$ estimate for the intermediate numerical error function at the previous time step is needed in the error analysis. Meanwhile, the \textit{a-priori} $H_h^1$ estimate for the numerical error for the magnetization vector can be controlled by a growth factor $1+\delta$ acting on the $H_h^1$ estimate of the intermediate magnetization error function, with $\delta$ being arbitrary positive number. Such a $W_h^{1,\infty}$ bound for numerical solution and $H_h^1$ estimate for the intermediate numerical error function can be recovered at the next time step, 
with the help of the inverse inequality and a mild temporal constraint and large damping parameter (larger than three). In turn, an estimate for numerical error function becomes a straightforward consequence of an application of discrete Gronwall inequality, combined with the fine estimate of a growth factor $1+\delta$ acting on the $H_h^1$ estimate of the intermediate error function. 

The rest of this paper is organized as follows. In Section~\ref{sec: numerical scheme}, we review the fully discrete numerical schemes and state the main theoretical result of convergence. The detailed proof is provided in Section~\ref{sec: convergence}. 
Finally, some concluding remarks are made in Section~\ref{sec:conclusions}.

\section{The mathematical model and the numerical scheme}
\label{sec: numerical scheme}

\subsection{The Landau-Lifshitz equation}
The LL equation is formulated as  
\begin{equation}\label{c1}
\begin{cases}
\partial_t{\m}(\x,t)=-{\m}\times\Delta{\m} + \alpha \Delta \m + \alpha | \nabla \m |^2 \m,&\x\in\Omega,\,t>0,\\
\partial_{\bf n}\m(\x,t)|_{\partial\Omega}=0,&\x\in\partial\Omega,\,t\ge0,\\
{\m}(\x,0)=\m_0(\x),&\x\in\Omega,
\end{cases}
\end{equation}
where $\x$ and $t$ are the variables of space and time, respectively, $\Omega\subset\mathbb{R}^d$ ($d=1,2,3$, with $d$ being the spatial dimension) is a bounded domain and  ${\bf n}$ is the unit outward normal vector along $\partial\Omega$, 
 ${\m}:=\m(\x,t)=(m_1,m_2,m_3)^T\,:\,\Omega\subset\mathbb{R}^d\to S^2$ represents the magnetization vector field with $|{\m}|=1,\;\forall \x\in\Omega$, and $\alpha>0$ is the damping parameter. The notations $\partial_t$, $\nabla$, and $\Delta$ represent the temporal derivative, the gradient and the Laplacian, respectively. The homogeneous Neumann boundary condition is considered. The first term on the right hand side of \cref{c1} is the gyromagnetic term, and the remaining term related to $\alpha$ is the damping term. In comparison with the ferromagnetic model~\cite{Landau1935On}, \cref{c1} only includes the exchange term which poses the main difficulties in numerical analysis, as done in the literature \cite{bartels2006convergence, weinan2001numerical, gao2014optimal}. 
To simplify the presentation, we set $\Omega=[0,1]^d$ and consider the 3D case in this paper, while the results hold for the 1D and 2D models.



\subsection{Finite difference discretization and the numerical method} \label{discretisations}
We set the temporal step-size as $k>0$, so that the time step instant becomes $t^n=nk$ ($n\leq \left\lfloor\frac{T}{k}\right\rfloor$, with $T$ being the final time, $\left\lfloor\cdot\right\rfloor$ being the floor operator). The spatial mesh-size is given by $h_x=\frac{1}{N_x}$, $h_y=\frac{1}{N_y}$, $h_z=\frac{1}{N_z}$, with $N_x$, $N_y$ and $N_z$ being the number of grid points of uniform partition along $x$, $y$ and $z$ directions, respectively. We use the  half grid points $({x}_{i-\frac{1}{2}},{y}_{j-\frac{1}{2}},{z}_{\ell-\frac{1}{2}})$ (also written as $(\hat{x}_i, \hat{y}_j, \hat{z}_\ell)$),  with $x_{i-\frac{1}{2}}=(i-\frac{1}{2})h_x$, $y_{j-\frac{1}{2}}=(j-\frac{1}{2})h_y$ and $z_{\ell-\frac{1}{2}}=(\ell-\frac{1}{2})h_z$ ( $ i=0,1,\cdots, N_x+1$; $j=0,1,\cdots,N_y+1$; $\ell=0,1,\cdots, N_z+1$). The numerical domain becomes $\Omega_h=\{(\hat{x}_i, \hat{y}_j, \hat{z}_\ell)| i=0,1,\cdots, N_x+1; j=0,1,\cdots, N_y+1; \ell=0,1,\cdots, N_z+1\}$ and the interior domain is $\Omega_h^0=\{(\hat{x}_i, \hat{y}_j, \hat{z}_\ell)| i=1,\cdots, N_x; j=1,\cdots, N_y; \ell=1,\cdots, N_z\}$, and $\Omega_h/\Omega_h^0$ is the set of boundary (ghost) points. 
We introduce the notation of the discrete vector grid function $\m_h(\x)\in\mathbb{R}^3$  defined for $\x\in\Omega_h$ with $\m_{i,j,\ell}=\m_h(\hat{x}_i,,\hat{y}_j,\hat{z}_\ell)$ (similar notations for the scalar functions ), and the discrete homogeneous Neumann boundary condition reads for $i_x=0,N_x$, $j_y=0,N_y$, $\ell_z=0,N_z$ and $ 0 \le i \le N_x+1, 0 \le j \le N_y+1, 0\le \ell \le N_z+1$,
 \begin{equation}\label{BC-1}
\m_{i_x,j,\ell}=\m_{i_x+1,j,\ell},\quad\m_{i,j_y,\ell}=\m_{i,j_y+1,\ell},\quad \m_{i,j,\ell_z}=\m_{i,j,\ell_z+1}.
\end{equation}
Let $X=\{f_h(\x)\in\mathbb{R},\x\in\Omega_h,\, f_h\, \mbox{satisfies boundary condition}~\eqref{BC-1} \}$ be the scalar function space and $\X=\{\m_h(\x)\in\mathbb{R}^3,\x\in\Omega_h,\, \m_h\, \mbox{satisfies boundary condition}~\eqref{BC-1} \}$ be the vector-valued function space. The corresponding continuous version is denoted by $X_e, \X_e$. The standard second-order centered difference approximation for the Laplace operator results in
\begin{equation*}
\Delta_h\m_{i,j,\ell}=\delta_x^2\m_{i,j,\ell}+\delta_y^2\m_{i,j,\ell}+\delta_z^2\m_{i,j,\ell},\quad \delta_x^2 \m_{i,j,\ell} =\frac{\m_{i+1,j,\ell}-2\m_{i,j,\ell}+\m_{i-1,j,\ell}}{h_x^2},
\end{equation*}
where  $\delta_y^2$, $\delta_z^2$ for the approximation of $\partial_{yy}$, $\partial_{zz}$ could be  similarly defined. The forward finite  difference operators $D_x,D_y$ and $D_z$ are defined for $f_h\in X$: 
\begin{equation*}
D_xf_{i,j,\ell}=\frac{f_{i+1,j,\ell}-f_{i,j,\ell}}{h_x},\, D_yf_{i,j,\ell}=\frac{f_{i,j+1,\ell}-f_{i,j,\ell}}{h_y},\,
D_zf_{i,j,\ell}=\frac{f_{i,j,\ell+1}-f_{i,j,\ell}}{h_z} . 
\end{equation*}
These finite difference operators could be applied to the scalar or vector grid functions in the same way. The discrete gradient operator (forward)  $\nabla_h \m_h$ with $\m_h=(u_h, v_h, w_h)^T\in \X$ reads as 
\begin{align*}
\nabla_h\m_{i,j,\ell} = \begin{bmatrix}
D_xu_{i,j,\ell}&D_xv_{i,j,\ell}&D_xw_{i,j,\ell}\\
D_yu_{i,j,\ell}&D_yv_{i,j,\ell}&D_yw_{i,j,\ell}\\
D_zu_{i,j,\ell}&D_zv_{i,j,\ell}&D_zw_{i,j,\ell}
\end{bmatrix}.
\end{align*}

A semi-implicit numerical scheme has been proposed in \cite{Xie2018}, and used in the numerical simulation for small damping parameter models. In more details, semi-implicit approximations are applied to the two nonlinear terms, namely $\m \times \Delta \m$ and $\m \times (\m \times \Delta \m)$, in which $\Delta \m$ is treated implicitly, while the coefficient variables are explicitly updated via an extrapolation formula. The theoretical convergence analysis has been established in a more recent work~\cite{jingrun2019analysis}. However, this numerical scheme involves a large linear system with time-dependent coefficients, related to the updated magnetization at each time step, and the symmetry is not available in the linear system, due to the nonlinear structure. To overcome this subtle difficulty, which leads to significant computational costs (especially in the 3D case), we make use of the alternate PDE formulation~\eqref{c1}, and treat the linear diffusion term $\alpha \Delta \m$ implicitly, while  the two nonlinear terms, namely $- \m \times \Delta \m$ and $\alpha | \nabla \m |^2 \m$, are discretized in a fully explicit way. Subsequently, a point-wise projection is applied to the intermediate field, so that the numerical solution of $\m$ has a unit length at the point-wise level. 
In more details, the following numerical scheme has been studied, namely {\bf Algorithm 2.1}: 
for given $\m_h,\tilde{\m}_h^n,\m_h^{n+1},\tilde{\m}_{h}^{n+1}\in \X$, find $\tilde{\m}_h^{n+2},\m_h^{n+2}\in \X$ by solving
\begin{align*}
\frac{\frac32 \tilde{\m}_h^{n+2}(\x) - 2 \tilde{\m}_h^{n+1}(\x) + \frac12 \tilde{\m}_h^n(\x)}{\dtk}
&=  - \hat{\m}_h^{n+2} \times \Delta_h \hat{\tilde{\m}}_h^{n+2} \\
& + \alpha \Delta_h \tilde{\m}_h^{n+2}  
+ \alpha | {\mathcal A}_h \nabla_h \hat{\m}_h^{n+2} |^2 \hat{\m}_h^{n+2} ,\quad \x\in\Omega_h^0,
\end{align*}
followed by the point-wise projection $\m_h^{n+2} = \frac{\tilde{\m}_h^{n+2}}{ |\tilde{\m}_h^{n+2}| }$,
where the extrapolation formula is defined by $\hat{\m}_h^{n+2} = 2 \m_h^{n+1} - \m_h^n , \, \, \, 
\hat{\tilde{\m}}_h^{n+2} = 2 \tilde{\m}_h^{n+1} - \tilde{\m}_h^n$, and ${\mathcal A}_h\nabla_h$ (second approximation to the gradient operator) is an average gradient operator defined for the gird function $\m_h=(u_h, v_h, w_h)^T\in \X$ as $\mathcal{A}_h\nabla_h\m_h=\nabla_h\mathcal{A}_h\m_h$ and $\mathcal{A}_h\m_h=(\mathcal{A}_xu_h,\mathcal{A}_yv_h,\mathcal{A}_zw_h)$: 
\begin{equation*}
\mathcal{A}_xu_{i,j,\ell}=\frac{u_{i,j,\ell}+u_{i-1,j,\ell}}{2},\,
\mathcal{A}_yv_{i,j,\ell}=\frac{v_{i,j,\ell}+v_{i,j-1,\ell}}{2},\,
\mathcal{A}_zw_{i,j,\ell}=\frac{w_{i,j,\ell}+w_{i,j,\ell-1}}{2}.
\end{equation*} 
%
In addition, a modified version is proposed by \cite{Xie2021BDF2Linear}, namely {\bf Algorithm 2.2}: for given $\m_h,\m_h^{n+1}\in \X$, denote $\hat{\m}_h^{n+2} = 2 \m_h^{n+1} - \m_h^n$, and find $\m_h^{n+2},\tilde{\m}_h^{n+2}\in \X$ by solving

\begin{align}
\frac{\frac32 \tilde{\m}_h^{n+2}(\x) - 2 \m_h^{n+1}(\x) + \frac12 \m_h^n(\x)}{\dtk}
&=  - \hat{\m}_h^{n+2} \times \Delta_h \hat{\m}_h^{n+2} + \alpha \Delta_h \tilde{\m}_h^{n+2}    \nonumber\\
\label{scheme-1-1}&\quad
 + \alpha | {\mathcal A}_h \nabla_h \hat{\m}_h^{n+2} |^2 \hat{\m}_h^{n+2} , \qquad\x\in\Omega_h^0,\\
 \m_h^{n+2} &= \frac{\tilde{\m}_h^{n+2}}{ |\tilde{\m}_h^{n+2}| } . \label{scheme-1-2}
\end{align}

\begin{remark}
The initial data is given by $\m_h^0=\mathcal{P}_h\m_0\in X$, where ${\mathcal P}_h:[C(\Omega)]^3\to X$ is the point-wise interpolation as
\begin{equation} 
  \mathcal{P}_h\m_0(\x)=  \m_0 (\x),\quad \x\in\Omega_h^0 . \label{initial data-1} 
\end{equation} 
In addition, the first-order semi-implicit projection scheme could be applied to obtain $\m_h^1$, so that the two-step numerical method could be jump started. Such a single-step first order algorithm will preserve the overall second-order accuracy in time; see the detailed analysis in the related works~\cite{guo16, guo2021} for Cahn-Hilliard equation, in which a single step, first order semi-implicit algorithm creates a second order accurate numerical solution in the first step.
\end{remark}

\begin{remark}
The primary difference between \textit{Algorithm 2.1} and {\bf Algorithm 2.2} is focused on the temporal-derivative approximation, where {\bf Algorithm 2.1} uses the intermediate approximate  magnetization, while \textit{Algorithm 2.2} takes the previous projected values. 
Extensive numerical experiments have demonstrated that, {\bf Algorithm 2.2} provides a much better stability than {\bf Algorithm 2.1} in the simulation of the realistic ferromagnetic material with large damping parameters, as reported in~\cite{Xie2021BDF2Linear}. In this article, we present a theoretical justification of the stability and convergence analysis for {\bf Algorithm 2.2}.  
\end{remark}


\subsection{Main theoretical results}

For simplicity of presentation, we make an assumption that $N_x = N_y =N_z=N$ (with $d=3$) so that $h_x = h_y = h_z =h$. An extension to the general case is straightforward. 
For the gird function $\f_h,\g_h\in\X$, the discrete $\ell^2$ inner product  $\langle\cdot\rangle$, discrete $\| \cdot \|_2$ norm  and discrete $\|\cdot\|_\infty$ norm are defined as 
 \begin{equation*} 
 \langle {\bm f}_h,{\bm g}_h \rangle = h^3\sum_{\x\in \Omega_h^0} \f_{h}(\x)\cdot\g_h(\x),
 \, \| {\bm f}_h \|_2 =  \sqrt{\langle {\bm f}_h,{\bm f}_h \rangle },\, \|\f_h\|_\infty=\max_{\x\in\Omega_h^0}|\f_h(\x)|.
\end{equation*}
In addition, the discrete $H_h^1$-norm is given by $\| \f_h \|_{H_h^1}^2 :=\|\f_h\|_2^2+\|\nabla_h \f_h\|_2^2$, and the discrete $\ell^p$ ($1<p<\infty$) norm is defined as $\| \f_h \|_{p}^p=h^3\sum_{\x\in \Omega_h^0}|\f_h(\x)|^p$. Such norms induce the discrete spaces
\begin{align*}
	\ell^p&=\{\f_h\in \X\;|\;\|\f_h\|_p<\infty\},\quad 0<p\le \infty\\
	H_h^1&=\{\f_h\in \X\;|\;\|\f_h\|_{H_h^1}<\infty\},\\
	W^{1,\infty}_h&=\{\f_h\in \X\;|\;\|\f_h\|_{\infty}+\|\nabla_h\f_h\|_{\infty}<\infty\},\\
	W^{1,4}_h&=\{\f_h\in \X\;|\;\|\f_h\|_{4}+\|\nabla_h\f_h\|_{4}<\infty\},\\
	\ell^{\infty}(0,T;\ell^2)&=\{\f_h^n\in \X\;|\;\max_n\|\f^n_h\|_2<\infty,\;n\in \left[0,\left\lfloor\frac{T}{k}\right\rfloor\right]\},\\
	\ell^{2}(0,T;H_h^1)&=\left\{\f_h^n\in \X\;\left|\right.\;\left(\sum_{n=0}^N\|\f^n_h\|_2^2\right)^{\frac12}<\infty,\;n\in \left[0,\left\lfloor\frac{T}{k}\right\rfloor\right]\right\}.
\end{align*}
Meanwhile, we define the continuous spaces for the function $\f(\x,t)=(f_1,f_2,f_3)$ as below, 
\begin{align*}
C^3([0,T];[C^0(\Omega)]^3)&:=\{\f(\x,t)\in \X_e\;|\; \frac{d^3}{dt^3} f_i \in C^0([0,T]),f_i\in C^0(\Omega),\;i=1,2,3\},\\
C^2([0,T];[C^2(\bar{\Omega})]^3)&:=\{\f(\x,t)\in \X_e\;|\;\frac{d^2}{dt^2} f_i \in C^0([0,T]),\frac{d^2}{d\x^2}f_i\in C^0(\bar{\Omega}),\;i=1,2,3\},\\
L^{\infty}([0,T];[C^4(\bar{\Omega})]^3)&:=\{\f(\x,t)\in \X_e\;|\;\textrm{ess sup}_{t\in[0,T]} \frac{d^4}{d\x^4}f_i\in C^0(\bar{\Omega}),\;i=1,2,3\},
\end{align*}
where $C^0(\Omega)$ is the space of continuous function.

The unique solvability  of scheme~\cref{scheme-1-1}-\cref{scheme-1-2}  follows from the equivalent form of \cref{scheme-1-1}:  
\begin{equation*} 
 ( \frac{3}{2 \dtk} I - \alpha \Delta_h ) \tilde{\m}_h^{n+2}(\x) 
 = \q_h^{n+2} (\x) ,\quad \x\in\Omega_h^0,
\end{equation*} 
where $\tilde{\m}_h^{n+2}\in\X$ and $\q_h^{n+2} := \frac{2 \m_h^{n+1} - \frac12 \m_h^n}{\dtk}
  - \hat{\m}_h^{n+2} \times \Delta_h \hat{\m}_h^{n+2}  
 + \alpha | {\mathcal A}_h \nabla_h \hat{\m}_h^{n+1} |^2 \hat{\m}^{n+2} $.
The left hand side corresponds to a positive-definite symmetric matrix,  and the unique solvability of the proposed scheme~\cref{scheme-1-1}-\cref{scheme-1-2}, as well as the \textit{Algorithm 2.1}, is obvious. With the fast discrete  Cosine transform, the  above linear system can be very efficiently solved. 

The theoretical results concerning the convergence analysis is stated below.
\begin{thm}\label{cccthm2} Assume that the exact solution of \cref{c1} has the regularity $\m_e \in C^3 ([0,T]; [C^0(\bar{}\Omega)]^3) \cap C^2([0,T]; [C^2(\bar{\Omega})]^3) \cap L^{\infty}([0,T]; [C^4(\bar{\Omega})]^3)$. Denote ${\m}_h^n$ ($n\ge0$) as the numerical solution obtained from~\cref{scheme-1-1}-\cref{scheme-1-2} with the initial  error satisfying $\|\mathcal{P}_h \m_e (\cdot,t_p) - \m_h^p \|_2 +\|\nabla_h ( \mathcal{P}_{h} \m_e (\cdot,t_p) - \m_h^p ) \|_2 = \mathcal{O} (\dtk^2 + h^2),\,p=0,1$. In addition, we assume  the technical assumption $\alpha > 3$, and $\mathcal{C}_1h\leq\dtk\leq \mathcal{C}_2h$, with $\mathcal{C}_1, \mathcal{C}_2$ being the positive constants. Then the following convergence result holds for $2\leq n\leq \left\lfloor\frac{T}{k}\right\rfloor$ as $h,\dtk\to0^+$:	\begin{align} \label{convergence-0} 
	\| \mathcal{P}_h \m_e (\cdot,t_n) - \m_h^n \|_{2} 
	+ \Big( \dtk \sum_{p=1}^n \|\nabla_h ( \mathcal{P}_{h} \m_e (\cdot,t_p)- \tilde{\m}_h^p ) \|_{2}^2 \Big)^\frac12  
	&\leq \mathcal{C}( \dtk^2+h^2) ,
	\end{align}	
	in which the constant $\mathcal{C}>0$ is independent of $\dtk$ and $h$.
\end{thm}

\subsection{A few preliminary estimates} 
\label{preliminary}

In this section, some preliminary inequalities are derived, which will be useful in the error analysis presented in the next section. In addition, we have to build a stability estimate of the projection step in the numerical algorithm. 

The proof of the standard inverse inequality and discrete Gronwall inequality can be obtained in existing textbooks and references; we just cite the results here. In the sequel, for simplicity of notation, we will use the uniform constant $\mathcal{C}$ to denote all the controllable constants.
\begin{lemm}(Inverse inequality) \cite{chen16, chenW20a, Ciarlet1978} \label{ccclemC1}.
	For each vector-valued grid function $\e_h\in X$, we have 
	\begin{align} \label{inverse-1}
	\|{\e}_h^{n}\|_{\infty} \leq \gamma {h}^{-1/2 } 
	( \|{\e}_h^{n}\|_2 + \| \nabla_h \e_h^n \|_2 ) ,  \quad
	\|\nabla_h{\e}_h^{n}\|_4 \leq \gamma {h}^{-3/4} \|\nabla_h{\e}_h^{n}\|_2 , 
	\end{align}
in which constant $\gamma$ depends on $\Omega$, as well as the form of the discrete $\| \cdot \|_2$ norm. 
\end{lemm}

The following estimate will be utilized in the convergence analysis. A rough version has been provided in a recent article~\cite{jingrun2019analysis}; here we give a further refined estimate. 

\begin{lemm}[Discrete gradient acting on cross product] \label{lem27}
For grid functions $\f_h,\g_h , \F_h \in \X$, we have for any $\delta>0$
	\begin{align}
	\left\langle \f_h\times \Delta_h \g_h , \F_h \right\rangle &=\left\langle \F_h \times \f_h ,\Delta_h\g_h\right\rangle , \label{lem27_1}  \\ 	
	\|\nabla_h({\f}_h \times{\g}_h ) \|_2^2 &\leq ( 1 + \delta) \|{\f}_h\|_{\infty}^2\cdot \|\nabla_h {\g}_h\|_2^2+ {\mathcal C}_\delta \|{\g}_h\|_4^2 \cdot \|\nabla_h \f_h\|_4^2.
	\label{lem27_2} 
	\end{align} 
\end{lemm}

\begin{proof} 
Equality~\eqref{lem27_1} has been proved in~\cite{jingrun2019analysis}, so that we only focus on the proof of~\eqref{lem27_2}. 

At each numerical mesh cell, from $(\hat{x}_i, \hat{y}_j, \hat{z}_\ell)$ to $(\hat{x}_{i+1}, \hat{y}_j, \hat{z}_\ell)$, the following expansion identity is valid: 
\begin{equation} 
\begin{aligned} 
  D_x ( \f_h\times \g_h )_{i, j, \ell} = & (A_x\f_h)_{i,j,\ell}  \times ( D_x \g_h )_{i,j,\ell}  
  +  (A_x\g_h)_{i,j,\ell}  \times ( D_x \f_h )_{i,j,\ell}  ,  
\\
   \mbox{with} \, \, \, 
  (A_x \f_h)_{i,j,\ell}  = & \frac12 ( ( \f_h)_{i,j,\ell}   +  ( \f_h)_{i+1,j,\ell}  ) . 
\end{aligned} 
  \label{expansion-1} 
\end{equation} 
In turn, we get the following expansion, over each mesh cell: 
\begin{equation} 
  D_x ( \f_h \times \g_h ) = ( A_x \f_h ) \times ( D_x \g_h ) 
  +  ( D_x \f_h ) \times ( A_x \g_h )  .  \label{expansion-2} 
\end{equation}   
Subsequently, a careful application of discrete H\"older inequality reveals that 
\begin{align} 
  \| ( A_x \f_h ) \times ( D_x \g_h ) \|_2 
  \le & \| A_x \f_h \|_\infty \cdot \| D_x \g_h \|_2   
    \le \| \f_h \|_\infty \cdot \| D_x \g_h \|_2  ,  \label{est-cross-1-1}  
\\
    \| ( D_x \f_h ) \times ( A_x \g_h ) \|_2 
  \le & \| D_x \f_h \|_4 \cdot \| A_x \g_h \|_4   
    \le \| D_x \f_h \|_4 \cdot \| \g_h \|_4  ,  \label{est-cross-1-2}     
\end{align} 
in which the fact that $\| A_x \f_h \|_\infty \le \| \f_h \|_\infty$, $\| A_x \g_h \|_4 \le \| \g_h \|_4$, has been applied. Then we get 
\begin{equation}
   \| D_x ( \f_h \times \g_h )  \|_2 \le \| \f_h \|_\infty \cdot \| D_x \g_h \|_2  
   + \| D_x \f_h \|_4 \cdot \| \g_h \|_4  .   \label{est-cross-2}     
\end{equation} 
The corresponding estimates in the $y$ and $z$ directions can be similarly derived, and the technical details are skipped for the sake of brevity. 
A combination of~\eqref{est-cross-2} and its counterparts in $y$ and $z$ directions leads to 
 \begin{equation}
\begin{aligned} 
   \| \nabla_h ( \f_h \times \g_h ) \|_2^2 \le & 
   \| \f_h \|_\infty^2 \cdot \| \nabla_h \g_h \|_2  
   + 3 \| \nabla_h \f_h \|_4^2 \cdot \| \g_h \|_4^2     
\\
  & 
  + 6 \| \f_h \|_\infty \cdot  \| \nabla_h \f_h \|_4 \cdot \| \g_h \|_4 
\\
   \le & 
   ( 1 + \delta ) \| \f_h \|_\infty^2 \cdot \| \nabla_h \g_h \|_2  
   + ( 3 + 9 \delta^{-1}) \| \nabla_h \f_h \|_4^2 \cdot \| \g_h \|_4^2  , 
\end{aligned} 
      \label{est-cross-4}     
\end{equation} 
for any $\delta >0$, in which the Cauchy inequality has been applied in the last step. Therefore, the nonlinear cross product estimate~\eqref{lem27_2} has been derived. This finishes the proof of Lemma~\ref{lem27}.          
\end{proof}

The following discrete Sobolev inequality has been derived in the existing works~\cite{guan17a, guan14a}, for the discrete grid function with periodic boundary condition; an extension to the discrete homogeneous Neumann boundary condition can be made in a similar manner. 
\begin{lemm}[Discrete Sobolev inequality] \cite{guan17a, guan14a} \label{lem: Sobolev-1}
	For a grid function $\f_h\in\X$,  we have the following discrete Sobolev inequality: 
	\begin{align} 
	\| \f_h \|_4 \le \mathcal{C} 	\| \f_h \|_2^\frac14 \cdot \| \f_h \|_{H_h^1}^\frac34 
	\le  \mathcal{C}  ( \| \f_h \|_2  	
	 +   	\| \f_h \|_2^\frac14 \cdot \| \nabla_h \f_h \|_2^\frac34 ) , \label{Sobolev-1} 
	\end{align}
in which the positive constant $\mathcal{C}$ only depends on the domain $\Omega$. 
\end{lemm}

The following stability estimates of the point-wise projection  \eqref{scheme-1-2} are crucial for the error analysis, and the proof could be found in Appendix A.

\begin{lemm}  \label{lem 6-0} Assume the continuous vector function $\m_e\in [C(\Omega)]^3$ satisfies a regularity requirement $\| \m_e \|_{W^{1, \infty}} \le C^*$ (with $C^*$ being a positive constant) and the point-wise constraint $| \m_e | = 1$. Denoting $\um_h = {\mathcal P}_h \m_e\in X$, 
for any grid function $\tilde{\m}_h\in X$, we define the projected grid function $\m_h\in X$ as $\m_h = \frac{\tilde{\m}_h}{ | \tilde{\m}_h |}$, and introduce the  error functions as $\e_h(\x) = \um_h(\x) - \m_h(\x)$, $\tilde{\e}_h(\x) = \um_h(\x) -\tilde{\m}_h(\x)$ ($\x\in\Omega_h$). Under the {\it a-priori} assumptions on  $\tilde{\e}_h$ or equivalently on the profile $\tilde{\m}_h$: 
\begin{equation} 
    \|  \tilde{\e}_h \|_2 \le 2 \dtk^{\frac{15}{8}} , \quad 
    \| \nabla_h \tilde{\e}_h \|_2 \le \frac12 \dtk^{\frac{11}{8}},  \label{a priori-0} 
\end{equation}      
 the following estimates hold for sufficiently small $\dtk$ and $h$ satisfying $C_1 h \le \dtk \le C_2 h$ (with $C_1, C_2$ being two positive constants) 
	\begin{align} 
	  & 
	 \| \tilde{\e}_h \|_2^2 \ge ( 1 - \dtk^\frac54 ) \| \e_h \|_2^2  
	 + ( 1 - \dtk^\frac14 ) \| \tilde{\e}_h - \e_h \|_2^2 ,   \label{lem 6-2}	
\\
      & 
	\| \nabla_h \e_h \|_2^2 \le ( 1 + \delta )  \| \nabla_h \tilde{\e}_h \|_2^2 
	+ \mathcal{C}_\delta \| \tilde{\e}_h \|_2^2  ,   \quad \mbox{for any $\delta>0$ } . 
	 \label{lem 6-2-2}
	\end{align}
\end{lemm}

In addition, for the analysis of the BDF2 temporal stencil at the projection stage, a further refined error estimate is needed. 

\begin{lemm}  \label{lem 7-0}
	Consider $\um_h^{(1)},\um_h^{(2)}\in X$  ($| \m_h^{(q)}(\x) | = 1$, $\x\in\Omega_h$, $q=1, 2$) 
	with the $W_h^{1, \infty}$ bound  $\| \um_h^{(q)} \|_\infty + \| \nabla_h \um_h^{(q)} \|_\infty \le C^*$ ($q=1,2$).
For any grid functions $\tilde{\m}_h^{(1)},\tilde{\m}_h^{(2)}\in X$, we define the projected grid functions $\m_h^{(q)}  = \frac{\tilde{\m}_h^{(q)}}{ | \tilde{\m}_h^{(q)} |}$, and introduce the error functions as $\e_h^{(q)}(\x) = \um_h^{(q)}(\x) - \m_h^{(q)}(\x)\in X$, $\tilde{\e}_h^{(q)}(\x) = \um_h^{(q)}(\x) -\tilde{\m}_h^{(q)}(\x)\in X$, $q=1, 2$. Under the {\it a-priori} assumptions for $\tilde{\e}_h^{(q)}$ ($q=1,2$): 
\begin{equation} 
    \|  \tilde{\e}_h^{(q)} \|_2 \le 2 \dtk^{\frac{15}{8}} , \quad 
    \| \nabla_h \tilde{\e}^{(q)} \|_2 \le \frac12 \dtk^{\frac{11}{8}} ,  \, \, \, 
    q= 1, 2 ,  \label{a priori-n-1} 
\end{equation}    
and the assumptions for $\um_h^{(q)}$ ($q=1,2$): 
\begin{equation} 
   \| \um_h^{(1)} - \um_h^{(2)} \|_\infty \le \frac14 \dtk^{\frac78},
   \label{a priori-n-2} 
\end{equation}    
the following estimate is valid for sufficiently small $\dtk$ and $h$ satisfying $C_1 h \le \dtk \le C_2 h$: 
	\begin{equation} 
	  \Big| \langle \tilde{\e}_h^{(1)} - \e_h^{(1)} ,  \e_h^{(2)} \rangle \Big| 	  	  
	  \le \dtk^\frac54  \| \e_h^{(2)} \|_2^2  
	 + \dtk^\frac14  \| \tilde{\e}_h^{(1)} - \e_h^{(1)} \|_2^2 .   \label{lem 7-1}	
	\end{equation}
\end{lemm}

We leave the proof of Lemma \ref{lem 7-0} to Appendix B. Lemmas \ref{lem 6-0} and \ref{lem 7-0} essentially establish the stability of the projection step \eqref{scheme-1-2} under the assumptions that the previous numerical solution at $t^{n}$ and $t^{n+1}$ are sufficiently close to the exact solution.

\section{The optimal rate convergence analysis: Proof of Theorem~\ref{cccthm2}} 
\label{sec: convergence} 

Denote $\underline{\m}(\x,t) = \mathcal{P}_h \m_e (\x,t)\in \X$ ($\x\in\Omega_h$) and $\underline{\m}_h^{n}(\x) =\underline{\m}_h(\x,t^n) $ ($n\ge0$). 
Around the boundary section $z=0$, we set $\hat{z}_0 = - \frac12 h$, $\hat{z}_1 = \frac12 h$, and we can extend the profile $\underline{\m}$ to the numerical ``ghost" points, according to the extrapolation formula~\cref{BC-1}:
	\begin{equation}
	\underline{\m}_{i,j,0}= \underline{\m}_{i,j,1} , \quad
	\underline{\m}_{i,j,N_z+1} = \underline{\m}_{i,j,N_z} ,  \label{exact-3}
	\end{equation}
and the extrapolation for other boundaries can be formulated in the same manner. The proof of such an extrapolation yields a higher order $\mathcal{O}(h^5)$ approximation, instead of the standard $\mathcal{O}(h^3)$ accuracy has been applied in \cite{jingrun2019analysis}. Also see the related works~\cite{STWW2003, Wang2000, Wang2004} in the existing literature. 

	
 	Performing a careful Taylor expansion for the exact solution around the boundary section $z=0$, combined with the mesh point values: $\hat{z}_0 = - \frac12 h$, $\hat{z}_1 = \frac12 h$, we get
 	\begin{align}
 	\m_e  (\hat{x}_i, \hat{y}_j, \hat{z}_0 )
 	&= \m_e (\hat{x}_i, \hat{y}_j, \hat{z}_1 )
 	- h \partial_z \m_e (\hat{x}_i, \hat{y}_j, 0 )
 	- \frac{h^3}{24} \partial_z^3 \m_e (\hat{x}_i, \hat{y}_j, 0 )
 	+   \mathcal{O}(h^5) \nonumber
 	\\
 	&= \m_e (\hat{x}_i, \hat{y}_j, \hat{z}_1 )
 	- \frac{h^3}{24} \partial_z^3 \m_e (\hat{x}_i, \hat{y}_j, 0 )
 	+   \mathcal{O}(h^5) ,  \label{exact-4}
 	\end{align}
 	in which the homogenous boundary condition has been applied in the second step. It remains to determine $\partial_z^3 \m_e (\hat{x}_i, \hat{y}_j, 0 )$, for which we use information from the rewritten PDE~\eqref{c1} and its derivatives. Applying $\partial_z$ to the first evolutionary equation in~\eqref{c1} along the boundary section $\Gamma_z: z=0$ gives
 \begin{eqnarray}
 \begin{aligned} 
   & 
   (m_1)_{zt} 
   - 2 \alpha ( m_1 ( \nabla m_1 \cdot \nabla (m_1)_z  + \nabla m_2 \cdot \nabla (m_2)_z
   + \nabla m_3 \cdot \nabla (m_3)_z )  ) 
 \\
   & 
   - \alpha | \nabla \m_e |^2 (m_1)_z  
   - \alpha ( (m_1)_{zxx} + (m_1)_{zyy} + \partial_z^3 m_1 )  
 \\
   =& 
   ( m_3 )_z \Delta m_2  
   + m_3 ( (m_2)_{zxx} + (m_2)_{zyy} + \partial_z^3 m_2 )      
 \\
   &  
   -  ( m_2 )_z \Delta m_3 
   - m_2  ( (m_3)_{zxx} + (m_3)_{zyy} + \partial_z^3 m_3 )  , \quad
  \mbox{on} \, \, \Gamma_z .  
 \end{aligned} 
  \label{scheme-BC-2}
 \end{eqnarray}
 The first, third terms, and the first two parts in the fourth term on the left-hand side of~\eqref{scheme-BC-2} disappear, due to the homogeneous Neumann boundary condition for $m_1$. For the second term on the left hand side, we observe that
 \begin{equation}
   \nabla m_1 \cdot \nabla (m_1)_z  = (m_1)_x \cdot (m_1)_{zx} + (m_1)_y \cdot (m_1)_{zy} 
   + (m_1)_z \cdot (m_1)_{zz}
   = 0 ,   \quad \mbox{on} \, \, \Gamma_z ,  \label{scheme-BC-3}
 \end{equation}
 since $(m_1)_z =0$ on the boundary section. Similar derivations could be made to the two other terms on the left hand side:
 \begin{equation}
   \nabla m_2 \cdot \nabla (m_2)_z  = 0 ,  \, \, \, \nabla m_3 \cdot \nabla (m_3)_z  = 0 ,
   \quad \mbox{on} \, \, \Gamma_z .  \label{scheme-BC-4}
 \end{equation}
 Meanwhile, on the right hand side of~\eqref{scheme-BC-2}, we see that the first and third terms, as well as the first two parts in the second and fourth terms, disappear, which comes from the homogeneous Neumann boundary condition for $m_2$ and $m_3$. Then we arrive at 
 \begin{equation} 
   \alpha  \partial_z^3 m_1  =  - m_3  \partial_z^3 m_2  + m_2  \partial_z^3 m_3 , 
   \quad \mbox{on} \, \, \Gamma_z .  \label{scheme-BC-5-1}
 \end{equation} 
 Similarly, we are able to derive the following equalities: 
 \begin{equation} 
 \begin{aligned} 
   & 
   \alpha  \partial_z^3 m_2  =  m_1  \partial_z^3 m_3  - m_3  \partial_z^3 m_1 ,  
   \quad \mbox{on} \, \, \Gamma_z  
 \\
  & 
   \alpha  \partial_z^3 m_3  =  m_2  \partial_z^3 m_1  - m_1  \partial_z^3 m_2 , 
   \quad \mbox{on} \, \, \Gamma_z .  
 \end{aligned} 
 \label{scheme-BC-5-2}
 \end{equation} 
 In turn, for any $\alpha >0$, we observe that the matrix $\left( \begin{array}{ccc} 
  \alpha & m_3    & - m_2 \\
  - m_3  & \alpha  & m_1  \\ 
  m_2     & - m_1  & \alpha 
 \end{array} \right)$ has a positive determinant, so that the linear system~\eqref{scheme-BC-5-1}-\eqref{scheme-BC-5-2} has only one trivial solution: 
 \begin{equation} 
   \partial_z^3 m_1  =  \partial_z^3 m_2 = \partial_z^3 m_3 = 0  , 
   \quad \mbox{on} \, \, \Gamma_z .  \label{scheme-BC-5-3}
 \end{equation} 
As a result, an ${\mathcal O} (h^5)$ consistency accuracy for the symmetric extrapolation is obtained: 
\begin{equation} 
\begin{aligned} 
	\m_e  (\hat{x}_i, \hat{y}_j, \hat{z}_0 )
	= \m_e (\hat{x}_i, \hat{y}_j, \hat{z}_1 )
	+   \mathcal{O}(h^5) , \;
\underline{\m}  (\hat{x}_i, \hat{y}_j, \hat{z}_0 )
	= \underline{\m} (\hat{x}_i, \hat{y}_j, \hat{z}_1 )
	+   \mathcal{O}(h^5) .  
\end{aligned} 
   \label{exact-5}
\end{equation}
In other words, the extrapolation formula~\cref{exact-3} is indeed $\mathcal{O}(h^5)$ accurate.	

Subsequently, a detailed calculation of Taylor expansion, in both time and space, leads to the following truncation error estimate:
	\begin{align} \label{consistency-2}
	\begin{aligned}
	&\frac{\frac{3}{2} \underline{\m}_h^{n+2}(\x) - 2\underline{\m}_h^{n+1}(\x)
		+ \frac{1}{2} \underline{\m}_h^n(\x)}{\dtk} \\&= - \hat{\underline{\m}}_h^{n+2} 
	\times \Delta_h \hat{\underline{\m}}_h^{n+2} + \tau^{n+2} + \alpha \Delta_h \underline{\m}_h^{n+2} 
	+ \alpha | {\mathcal A}_h \nabla_h \hat{\underline{\m}}_h^{n+2} |^2 
	\hat{\underline{\m}}_h^{n+2} , \, \x \in\Omega_h^0,
	\end{aligned}
	\end{align}
where $\hat{\underline{\m}}_h^{n+2} = 2 \underline{\m}_h^{n+1} - \underline{\m}_h^n\in \X$, $\tau^{n+2}\in \X$ and $\| \tau^{n+2} \|_2 \le \mathcal{C} (\dtk^2+h^2)$. Introducing the numerical error functions 
$\tilde{\e}_h^n=\underline{\m}_h^n-\tilde{\m}_h^n\in\X$, ${\e}_h^n=\underline{\m}_h^n-\m_h^n\in \X$, and subtracting  \cref{scheme-1-1}-\cref{scheme-1-2} from the consistency estimate~\cref{consistency-2}, we have the error evolutionary equation at the interior points $\x\in\Omega_h^0$, for $0\leq n\leq \lfloor T/\dtk \rfloor-2$:
	\begin{align}\label{ccc73}
	\begin{aligned}
\frac{\frac{3}{2}\tilde{\e}_h^{n+2}-2 \e_h^{n+1} + \frac{1}{2} \e_h^n}{\dtk} &= 
- \hat{\m}_h^{n+2} \times \Delta_h \hat{\e}_h^{n+2} 
-\left( 2{\e}_h^{n+1}-{\e}_h^n\right) \times \Delta_h \hat{\um}_h^{n+2}\\
& + \alpha \Delta_h \tilde{\e}_h^{n+2} 
+ \alpha | {\mathcal A}_h \nabla_h \hat{\underline{\m}}_h^{n+2} |^2 
\hat{\e}_h^{n+2} + \tau^{n+2} \\
& + \alpha \Big( 
{\mathcal A}_h \nabla_h  ( \hat{\underline{\m}}_h^{n+2} + \hat{\m}_h^{n+2} ) 
\cdot {\mathcal A}_h \nabla_h  \hat{\e}_h^{n+2}  \Big) 		 
\hat{\m}_h^{n+2}
	\end{aligned}
	\end{align}
with $\hat{\e}_h^{n+2} = 2 \e_h^{n+1} - \e_h^n\in \X$. 
	
	Before proceeding into the formal error estimate, we state the bound for the exact solution ${\um}$ and the numerical solution $\m_h$. Since the exact solution $\m_e \in L^{\infty}([0,T]; [C^4(\bar{\Omega})]^3)$, the following bound is available, for some positive constant $\mathcal{C}$: 
	\begin{align}
	\|\nabla_h^r \um_h(\cdot,t) \|_4 , \, \, \|\nabla_h^r \um_h(\cdot,t) \|_{\infty} 
	\leq \mathcal{C}, \quad r=0,1,2,3,\, 0\leq t\leq T.  \label{bound-1}
	\end{align}
In addition, we make the following \textit{a-priori} assumption for the numerical error function:
	\begin{equation} \label{bound-2}
	\|{\e}_h^n \|_2 \le \dtk^{\frac{31}{16}} , \, \, \,  
	\|  \tilde{\e}_h^n \|_2 \le 2 \dtk^{\frac{15}{8}} , \, \, \, 
    \| \nabla_h \tilde{\e}_h^n \|_2 \le \frac12 \dtk^{\frac{11}{8}} , 	
	\, \, \,  \mbox{$n\leq q+1$} .
	\end{equation}
	Such an assumption will be recovered by the convergence analysis at time step $t^{q+2}$. Based on this \textit{a-priori} assumption, we see that~\eqref{a priori-0} is satisfied, so that we are able to apply Lemma~\ref{lem 6-0} and the estimate~\eqref{lem 6-2-2} to get	
	\begin{align} 
	\| \nabla_h \e_h^n \|_2 \le \frac54 \| \nabla_h \tilde{\e}_h^n \|_2 
	+ \mathcal{C} \| \tilde{\e}_h^n \|_2  
	\le \frac58 \dtk^{\frac{11}{8}}  + {\mathcal C} \dtk^{\frac{15}{8}} 
	\le \dtk^{\frac{11}{8}} ,  \, \, \,  \mbox{$n\leq q+1$} .	 
	 \label{bound-3-1}
	\end{align}
In turn, an application of inverse inequality implies the $\| \cdot \|_\infty$ and $\| \cdot \|_{W_h^{1,4}}$ bounds of the numerical error function $\e_h^n$ ($n\leq q+1$): 
	\begin{equation} 
\begin{aligned} 
    \| \e_h^n \|_\infty & \le \frac{C \| \e_h^n \|_2}{h^\frac32} 
   \le \frac{C \cdot \dtk^{\frac{31}{16}}}{h^\frac32}  \le C \dtk^{\frac{7}{16}} \le \dtk^\frac38 ,  
\\
   \| \nabla_h \e_h^n \|_4 & \le \frac{C \| \nabla_h \e_h^n \|_2}{h^\frac34} 
   \le \frac{C \cdot \dtk^{\frac{11}{8}}}{h^\frac34}  \le C \dtk^{\frac{5}{8}} \le \dtk^\frac12 \le \frac13 .      
\end{aligned} 
     \label{bound-4} 
     	\end{equation} 
Subsequently, the triangle inequality yields the desired $W_h^{1,4}$ bound for the numerical solutions $\m_h^n$ and $\tilde{\m}_h^n$ ($n\leq q+1$):
	\begin{align}
	\|\nabla_h{\m}_h^n \|_4 &= \|\nabla_h\um_h^n -\nabla_h{\e}_h^n \|_4 \leq 
	\|\nabla_h\um_h^n \|_4 +\|\nabla_h{\e}_h^n \|_4 
	\leq \mathcal{C}+\frac13 .  \label{bound-3}  			
	\end{align}
Furthermore, we need a sharper $\| \cdot \|_\infty$ bound for $\hat{\m}_h^{n+2} = 2 \m_h^{n+1} - \m_h^n$, which will be needed in the later error analysis. The following extrapolation estimate is valid, due to the $C^3 ([0,T]; [C^0(\Omega)]^3)$ regularity of the exact solution $\m(\cdot,t)$: 
\begin{equation} 
  \um_h^{n+2} = 2 \um_h^{n+1} - \um_h^n + {\mathcal O} (\dtk^2) .  
  \label{bound-5-1} 
\end{equation} 
Meanwhile, since $| \m(\x,t) | \equiv 1$ ($\x\in\Omega$), we conclude that 
\begin{equation} 
  \| 2 \um_h^{n+1} - \um_h^n \|_\infty  \le 1 + {\mathcal C} \dtk^2 ,\quad n\leq q+1.  \label{bound-5-2} 
\end{equation}  
Its combination with the \textit{a-priori} assumption that $\|{\e}_h^n \|_{\infty} \le \dtk + h$, for $n\leq q+1$, (as given by~\eqref{bound-2}), implies that 
\begin{equation} 
\begin{aligned} 
  \| \hat{\m}_h^{n+2} \|_\infty = & \| 2 \m_h^{n+1} - \m_h^n \|_\infty  
  \le \| 2 \um_h^{n+1} - \um_h^n \|_\infty  + \| 2 \e_h^{n+1} - \e_h^n \|_\infty    
\\
  \le & 1 + {\mathcal C} \dtk^2 + 3 \dtk^\frac38  
  \le \alpha_1 := \Big( \frac{3 + \alpha}{6} \Big)^\frac12 , 
\end{aligned} 
\label{bound-5-3}  
\end{equation}   
provided that $\dtk$ and $h$ are sufficiently small. In addition, we denote $\gamma_0 := \alpha-3 >0$, so that $\alpha_1^2 = 1 + \frac16 \gamma_0$. 
	
	Next, we perform a discrete $\ell^2$ error estimate at $t^{q+2}$ using the mathematical induction. 
	Taking a discrete inner product with the numerical error equation \cref{ccc73} by $\tilde{\e}_h^{n+2}\in\X$  ($n\leq q+1$) gives that
	\begin{align}\label{rhs}
	\begin{aligned}
R.H.S.=& \left\langle-\left( 2{\m}_h^{n+1}-{\m}_h^n\right)\times \Delta_h \hat{\e}_h^{n+2}, \tilde{\e}_h^{n+2} \right\rangle\\
&-\left\langle\left( 2{\e}_h^{n+1}-{\e}_h^n\right)\times\Delta_h \hat{\underline{\m}}_h^{n+2}, \tilde{\e}_h^{n+2}\right\rangle + \left\langle \tau^{n+2}, \tilde{\e}_h^{n+2} \right\rangle 
\\
& - \alpha \| \nabla_h \tilde{\e}_h^{n+2} \|_2^2 
+ \alpha \langle | {\mathcal A}_h \nabla_h \hat{\underline{\m}}_h^{n+2} |^2 
\hat{\e}_h^{n+2} , \tilde{\e}_h^{n+2} \rangle  \\
& + \alpha \Big\langle \Big( 
{\mathcal A}_h \nabla_h  ( \hat{\underline{\m}}_h^{n+2} + \hat{\m}_h^{n+2} ) 
\cdot {\mathcal A}_h \nabla_h  \hat{\e}_h^{n+2}  \Big) 		 
\hat{\m}_h^{n+2} ,  \tilde{\e}_h^{n+2} \Big\rangle   \\
=:& \tilde{I}_1+\tilde{I}_2+\tilde{I}_3 - \alpha \| \nabla_h \tilde{\e}_h^{n+2} \|_2^2  
+\tilde{I}_5+\tilde{I}_6.
	\end{aligned}
	\end{align}
Then all the terms are accordingly analyzed. For the term $\tilde{I}_1$, a combination of the summation by parts formula by and Cauchy inequality  results in 	
		\begin{align}\label{I1-1}
		\begin{aligned}
		\tilde{I}_1 =& \left\langle-\hat{\m}_h^{n+2}\times\Delta_h \hat{\e}_h^{n+2},\tilde{\e}_h^{n+2} \right\rangle
	=  \left\langle \tilde{\e}_h^{n+2}\times\hat{\m}_h^{n+2} , -\Delta_h \hat{\e}_h^{n+2}\right\rangle \\
	=&  \left\langle \nabla_h\Big[\tilde{\e}_h^{n+2}\times\hat{\m}_h^{n+2}\Big] , \nabla_h \hat{\e}_h^{n+2} \right\rangle 
	\le  \frac32 \Big\| \nabla_h \Big( \tilde{\e}_h^{n+2} 
	\times \hat{\m}_h^{n+2}  \Big) \Big\|_2^2 
	+ \frac16 \| \nabla_h \hat{\e}_h^{n+2} \|_2^2 . 
		\end{aligned}
		\end{align}
Meanwhile, an application of the cross product gradient estimate~\eqref{lem27_2} implies that, for any $\delta>0$, the following inequality is valid: 
\begin{equation} 
	\begin{aligned}	
	  & 
	\Big\| \nabla_h \Big( \tilde{\e}_h^{n+2} 
 		\times \hat{\m}_h^{n+2}  \Big) \Big\|_2^2 
\\
     \le  & 
		( 1 + \delta) \| \hat{\m}_h^{n+2} \|_{\infty}^2 
		\cdot \|\nabla_h \tilde{\e}_h^{n+2}  \|_2^2 
		+ {\mathcal C}_\delta  \| \tilde{\e}_h^{n+2}  \|_4^2 
		\cdot \|\nabla_h \hat{\m}_h^{n+2} \|_4^2 
\\
   \le  & 
		( 1 + \delta) \alpha_1^2  \|\nabla_h \tilde{\e}_h^{n+2}  \|_2^2 
		+ {\mathcal C}_\delta  \| \tilde{\e}_h^{n+2}  \|_4^2 ,      
\end{aligned} 
  \label{I1-2} 
	\end{equation} 
 where $\mathcal{C}_\delta$ is a positive constant dependent on $\delta$, in which the \textit{a-priori} bound estimates~\eqref{bound-4} and \eqref{bound-5-3} have been applied.   The term $\| \nabla_h \hat{\e}_h^{n+2} \|_2^2$ can be analyzed as follows: 
\begin{equation} 
\begin{aligned} 
    \| \nabla_h \hat{\e}_h^{n+2} \|_2^2 = & 
    \| \nabla_h ( 2 \e_h^{n+1} - \e_h^n ) \|_2^2 
    \le  6 \| \nabla_h \e_h^{n+1} \|_2^2 + 3 \| \nabla_h \e_h^n \|_2^2 
\\
   \le & 
    ( 1 + \delta)  ( 6 \| \nabla_h \tilde{\e}_h^{n+1} \|_2^2 
      + 3 \| \nabla_h \tilde{\e}_h^n \|_2^2  ) 
       + {\mathcal C}_\delta ( \| \tilde{\e}_h^{n+1} \|_2^2  
       + \| \tilde{\e}_h^n\|_2^2 )   ,  
\end{aligned} 
  \label{I1-3} 
\end{equation}     
in which the estimate~\eqref{lem 6-2-2} (in Lemma~\ref{lem 6-0}) has been repeatedly applied, due to the \textit{a-priori} assumption~\eqref{bound-2}. Combining ~\eqref{I1-2}, \eqref{I1-3} and \eqref{I1-1},  we get
		\begin{align}\label{I1}
		\begin{aligned}
	\tilde{I}_1 
	\le \, & \frac32 \Big\| \nabla_h \Big( \tilde{\e}_h^{n+2} 
	\times \hat{\m}_h^{n+2}  \Big) \Big\|_2^2 
	+ \frac16 \| \nabla_h \hat{\e}_h^{n+2} \|_2^2  \\ 	
	\le \, & ( 1 + \delta)  \Big( \frac32 \alpha_1^2  
	\|\nabla_h \tilde{\e}_h^{n+2}  \|_2^2 
	+  \| \nabla_h \tilde{\e}_h^{n+1} \|_2^2 
	+ \frac12 \| \nabla_h \tilde{\e}_h^n \|_2^2 \Big) \\		 
	&  + {\mathcal C}_\delta  ( \| \tilde{\e}_h^{n+2}  \|_4^2 		
	+ \| \tilde{\e}_h^{n+1} \|_2^2  
	+ \| \tilde{\e}_h^n \|_2^2 )  . 
		\end{aligned}
		\end{align}

For the term $\tilde{I}_2$, by the preliminary estimate~\eqref{bound-1} for the exact solution, we have
		\begin{align}\label{I2}
		\begin{aligned}
		\tilde{I}_2 =\, & -\left\langle\hat{\e}_h^{n+2}\times \Delta_h \hat{\um}_h^{n+2}, \tilde{\e}_h^{n+2}\right\rangle 		\le \, \frac{1}{2}\big[\|\tilde{\e}_h^{n+2}\|_2^2+\|\hat{\e}_h^{n+2}\|_2^2\cdot \|\Delta_h \hat{\um}_h^{n+2}\|_{\infty}^2 \big]\\
	\le \, & \mathcal{C} ( \|\tilde{\e}_h^{n+2}\|_2^2 + \| \e_h^{n+1}\|_2^2
	+ \| \e_h^{n}\|_2^2 ). 
		\end{aligned}
		\end{align}

	For the term $\tilde{I}_3$, an application of Cauchy inequality gives
		\begin{align}\label{I3}
		\tilde{I}_3 = \left\langle\tau^{n+2},\tilde{\e}_h^{n+2} \right\rangle \leq \mathcal{C}\|\tilde{\e}_h^{n+2}\|_2^2+\mathcal{C}(\dtk^4+h^4).
		\end{align}
			
In terms of $\tilde{I}_5$, based on the $W_h^{1, \infty}$ bound~\eqref{bound-1} for the exact solution, an application of discrete H\"older inequality gives 
          	\begin{align}\label{I5-1}
 		\Big\| | {\mathcal A}_h \nabla_h \hat{\underline{\m}}_h^{n+2} |^2 
 	\hat{\e}_h^{n+2} \Big\|_2 
 		\le \, & \| \nabla_h \hat{\underline{\m}}_h^{n+2} \|_\infty^2 
 		\cdot  \| \hat{\e}_h^{n+2} \|_2   \le   {\mathcal C}  \| \hat{\e}_h^{n+2} \|_2 , 
         	\end{align}		
and		
 	\begin{align}\label{I5}
 	\begin{aligned}
 \tilde{I}_5 = \,& \alpha \langle | {\mathcal A}_h \nabla_h \hat{\underline{\m}}_h^{n+2} |^2 
 \hat{\e}_h^{n+2} , \tilde{\e}_h^{n+2} \rangle \le  \alpha \Big\| | {\mathcal A}_h \nabla_h \hat{\underline{\m}}_h^{n+2} |^2 
 \hat{\e}_h^{n+2} \Big\|_2 \cdot \| \tilde{\e}_h^{n+2}	 \|_2  \\
 \le \, &  {\mathcal C} \alpha  \| \hat{\e}_h^{n+2} \|_2  
 \cdot \| \tilde{\e}_h^{n+2} \|_2 
 \le {\mathcal C} ( \| \e_h^{n+1}\|_2^2 + \| \e_h^{n} \|_2^2		
 + \|\tilde{\e}_h^{n+2}\|_2^2 ) . 
 	\end{aligned}
 	\end{align}
	
For the term $\tilde{I}_6$, an application of discrete H\"older inequality gives 		
		\begin{align} \label{I6}
		\begin{aligned}
	\tilde{I}_6 = \,& \alpha \Big\langle \Big( 
	{\mathcal A}_h \nabla_h  ( \hat{\underline{\m}}_h^{n+2} + \hat{\m}_h^{n+2} ) 
	\cdot {\mathcal A}_h \nabla_h  \hat{\e}_h^{n+2}  \Big) 		 
	\hat{\m}_h^{n+2} ,  \tilde{\e}_h^{n+2} \Big\rangle \\
	\le \, & \alpha \Big(  \| \nabla_h  \hat{\underline{\m}}_h^{n+2} \|_4  
	+ \| \nabla_h \hat{\m}_h^{n+2} \|_4 \Big) \cdot  
	\| \nabla_h  \hat{\e}_h^{n+2} \|_2  \cdot \| \hat{\m}_h^{n+2} \|_\infty 
	\cdot 	\| \tilde{\e}_h^{n+2}	 \|_4  \\ 
	\le \, & {\mathcal C} \alpha \| \nabla_h  \hat{\e}_h^{n+2} \|_2 
	\cdot 	 \| \tilde{\e}_h^{n+2}  \|_4                 
	\le \mathcal{C} \alpha^2 \gamma_0^{-1}  \|\tilde{\e}_h^{n+2}\|_4^2 
	+ \frac{\gamma_0}{72}  \| \nabla_h \hat{\e}_h^{n+2}\|_2^2   \\
	\le \, & \mathcal{C}  \|\tilde{\e}_h^{n+2}\|_4^2 
	+ \frac{( 1+ \delta) \gamma_0}{24}
	( 2 \| \nabla_h \tilde{\e}_h^{n+1} \|_2^2 
	+   \| \nabla_h \tilde{\e}_h^n \|_2^2  ) 
	+ {\mathcal C}_\delta ( \| \tilde{\e}_h^{n+1} \|_2^2  
	+ \| \tilde{\e}_h^n \|_2^2 ) ,  
		\end{aligned}	
		\end{align}
in which the $W_h^{1, 4}$ bounds~\eqref{bound-1}, \eqref{bound-3}, for the exact and numerical solutions, as well as the preliminary error estimate~\eqref{I1-3}, have been applied. 	
	On the other hand, the inner product of the left hand side of \cref{ccc73} with $\tilde{\e}_h^{n+2}$ turns out to be	
	\begin{align*}
	L.H.S. =& \frac{1}{4 \dtk}\big(\|\tilde{\e}_h^{n+2}\|_2^2-\| \e_h^{n+1}\|_2^2 
	+ \| 2 \tilde{\e}_h^{n+2} - \e_h^{n+1} \|_2^2 
	- \| 2 \e_h^{n+1} - \e_h^n \|_2^2 \\
	&+\| \tilde{\e}_h^{n+2} - 2 \e_h^{n+1} + \e_h^n\|_2^2 \big). \nonumber
	\end{align*}
	Its combination with \cref{I1,I2,I3,I5,I6} and \cref{rhs} leads to
	\begin{align}\label{convergence-1}
	\begin{aligned}
	&  \frac{1}{4 \dtk}(\|\tilde{\e}_h^{n+2}\|_2^2-\| \e_h^{n+1}\|_2^2 
	+ \| 2 \tilde{\e}_h^{n+2} - \e_h^{n+1} \|_2^2 
	- \| 2 \e_h^{n+1} - \e_h^n \|_2^2 ) \\ 
	&+  \alpha_2		 \|\nabla_h \tilde{\e}_h^{n+2}  \|_2^2 
	- ( 1 + \delta) ( 1 + \frac{\gamma_0}{12})  \Big ( 
	\| \nabla_h \tilde{\e}_h^{n+1} \|_2^2 
	+ \frac12 \| \nabla_h \tilde{\e}_h^n \|_2^2 \Big)   \\       
	&\le  \mathcal{C}_\delta \Big( \sum\limits_{p=n}^{n+2}	 
	\|\tilde{\e}_h^{p}\|_2^2  + \| \e_h^{n+1}\|_2^2 + \| \e_h^{n}\|_2^2 
	+ \| \tilde{\e}_h^{n+2}\|_4^2 \Big) 
	+ \mathcal{C} (\dtk^4+h^4), 
	\end{aligned}
	\end{align}
	where $\alpha_2=( \alpha - \frac32 \alpha_1^2 ( 1 + \delta) )$.
Meanwhile, for the $\| \cdot \|_4$ error estimate for $\tilde{\e}_h^{n+2}$, an application of the discrete Sobolev inequality~\eqref{Sobolev-1} (in Lemma~\ref{lem: Sobolev-1}) gives 
\begin{equation} 
	\begin{aligned} 
	{\mathcal C}_\delta \| \tilde{\e}_h^{n+2} \|_4^2 
	\le  & \mathcal{C}_\delta  ( \| \tilde{\e}_h^{n+2} \|_2^2  	
	 +   	\| \tilde{\e}_h^{n+2} \|_2^\frac12 
	 \cdot \| \nabla_h \tilde{\e}_h^{n+2} \|_2^\frac32 ) 
\\
	 \le  & \mathcal{C}_\delta  \| \tilde{\e}_h^{n+2} \|_2^2  	
	 +   	\frac{\gamma_0}{12} \| \nabla_h \tilde{\e}_h^{n+2} \|_2^2 ,  
	\end{aligned} 
	\label{convergence-1-2} 	
\end{equation} 
in which the Young's inequality has been applied. Then we get 
\begin{equation} 
	\begin{aligned}
	&  \frac{1}{4 \dtk}\big(\|\tilde{\e}_h^{n+2}\|_2^2-\| \e_h^{n+1}\|_2^2 
	+ \| 2 \tilde{\e}_h^{n+2} - \e_h^{n+1} \|_2^2 
	- \| 2 \e_h^{n+1} - \e_h^n\|_2^2 ) \\ 
           & + \Big( \alpha_2 - \frac{\gamma_0}{12} \Big)   
		 \|\nabla_h \tilde{\e}_h^{n+2}  \|_2^2   - ( 1 + \delta) ( 1 + \frac{\gamma_0}{12})  \Big ( 
		    \| \nabla_h \tilde{\e}_h^{n+1} \|_2^2 
                  + \frac12 \| \nabla_h \tilde{\e}_h^n\|_2^2 \Big)   \\       
	\le \, & \mathcal{C}_\delta \Big( 	 
	\|\tilde{\e}_h^{n+2}\|_2^2 + \|\tilde{\e}_h^{n+1} \|_2^2 
	+ \|\tilde{\e}_h^n\|_2^2 + \| \e_h^{n+1}\|_2^2 + \| \e_h^{n}\|_2^2 \Big) 
	+ \mathcal{C} (\dtk^4+h^4) 	. 
	\end{aligned} 
	\label{convergence-2} 
\end{equation}  
In particular, we observe that 
\begin{equation} 
\begin{aligned} 
  & 
  \big( \alpha - \frac32 \alpha_1^2 ( 1 + \delta) - \frac{\gamma_0}{12} \Big)   
              - \frac32 ( 1 + \delta) ( 1 + \frac{\gamma_0}{12})	
\\
              \ge  &  \alpha - 3 \alpha_1^2 ( 1 + \delta) 
       \ge  
          ( 3 + \gamma_0 ) - 3 ( 1 + \frac16 \gamma_0 ) ( 1 + \delta)   
          \ge \frac14 \gamma_0 , 
\end{aligned} 
  \label{condition-alpha-1} 
\end{equation}  
if $\delta >0$ is chosen with $( 1 + \frac16 \gamma_0 ) ( 1 + \delta)  \le 1 + \frac14 \gamma_0$. 

Moreover, an application of the \textit{a-priori} assumption~\eqref{bound-2} into~\eqref{convergence-2} yields  
\begin{equation} 
	\begin{aligned}
	&  ( \frac{1}{4 \dtk} - {\mathcal C}_\delta ) \|\tilde{\e}_h^{n+2}\|_2^2 
           + \Big( \frac32 ( 1 + \delta) ( 1 + \frac{\gamma_0}{12}) + \frac{\gamma_0}{4} \Big) 
		 \|\nabla_h \tilde{\e}_h^{n+2}  \|_2^2  
\\
           \le \, &  ( 1 + \delta) ( 1 + \frac{\gamma_0}{12})  \Big ( 
		    \| \nabla_h \tilde{\e}_h^{n+1} \|_2^2 
                  + \frac12 \| \nabla_h \tilde{\e}_h^n \|_2^2 \Big)         
	         +  \frac{1}{4 \dtk} \big(  \| \e_h^{n+1}\|_2^2 
	 + \| 2 \hat{\e}_h^{n+2}  \|_2^2 ) 
\\	
	 & + \mathcal{C}_\delta \Big( 	 
	   \|\tilde{\e}_h^{n+1} \|_2^2 
	+ \|\tilde{\e}_h^n\|_2^2 + \| \e_h^{n+1}\|_2^2 + \| \e_h^{n}\|_2^2 \Big) 
	+ \mathcal{C} (\dtk^4+h^4) 	
\\
      \le & 
        \frac32 ( 1 + \delta) ( 1 + \frac{\gamma_0}{12})  \cdot \frac14 \dtk^{\frac{11}{4}} 
        + {\mathcal C} ( \dtk^{\frac{23}{8}} + \dtk^{\frac{15}{4}} + \dtk^4 + h^4 )      
\\
      \le & 
        \Big( \frac32 ( 1 + \delta) ( 1 + \frac{\gamma_0}{12}) + \frac{\gamma_0}{8} \Big) 
        \cdot \frac14 \dtk^{\frac{11}{4}} ,         
	\end{aligned} 
	\label{convergence-3} 
\end{equation}  
provided that $\dtk$ and $h$ are sufficiently small, and under the linear refinement requirement, $C_1 h \le \dtk \le C_2 h$. As a matter of fact, we can choose $\gamma_0 >0$ and $\delta >0$ and make $\dtk$ sufficiently small so that 
\begin{equation}
   \frac32 ( 1 + \delta) ( 1 + \frac{\gamma_0}{12}) + \frac{\gamma_0}{8} \le 2 , \quad 
   \frac{1}{4 \dtk} - {\mathcal C}_\delta \ge \frac{1}{6 \dtk} .
     \label{convergence-3-2} 
\end{equation}          
As a result, by~\eqref{convergence-3}, we arrive at 
\begin{equation} 
	\begin{aligned}
	&  \frac{1}{6 \dtk} \|\tilde{\e}_h^{q+2}\|_2^2  \le  \frac12 \dtk^{\frac{11}{4}} , \quad
		 \|\nabla_h \tilde{\e}_h^{q+2}  \|_2^2  \le \frac14 \dtk^{\frac{11}{4}}, 
	\\
            &  \mbox{i.e.},  \quad 
           \|\tilde{\e}_h^{q+2}\|_2  \le  \sqrt{3} \dtk^{\frac{15}{8}}   \le 2  \dtk^{\frac{15}{8}} , 
           \quad   \|\nabla_h \tilde{\e}_h^{q+2}  \|_2  \le \frac12 \dtk^{\frac{11}{8}}  ,         
	\end{aligned} 
	\label{convergence-3-3} 
\end{equation}  
so that the \textit{a-priori} assumption is valid for $\tilde{\m}^{q+2}$. 

With the recovery of the \textit{a-priori} estimate~\eqref{convergence-3-3} at time step $t^{q+2}$,  we are able to apply estimates~\eqref{lem 6-2}, \eqref{lem 7-1} (in Lemmas~\ref{lem 6-0} and \ref{lem 7-0}), respectively: 
	\begin{align} 
	  & 
	 \| \tilde{\e}_h^{n+2} \|_2^2 \ge ( 1 - \dtk^\frac54 ) \| \e_h^{n+2} \|_2^2  
	 + ( 1 - \dtk^\frac14 ) \| \tilde{\e}_h^{n+2} - \e_h^{n+2} \|_2^2 ,   
	  \label{convergence-4-1}	
\\
          & 
	  \Big| \langle \tilde{\e}_h^{n+2} - \e_h^{n+2} ,  \e_h^{p} \rangle \Big| 	  	  
	  \le \dtk^\frac54  \| \e_h^{p} \|_2^2  
	 + \dtk^\frac14  \| \tilde{\e}_h^{n+2} - \e_h^{n+2} \|_2^2 , \, \, 
	 p = n+1, n+2.    \label{convergence-4-2}	
	\end{align}
Moreover, the following inequality becomes available for $n\leq q$: 
\begin{equation} 
\begin{aligned} 
    \| 2 \tilde{\e}_h^{n+2} - \e_h^{n+1} \|_2^2 
    = & \| \hat{ \e}_h^{n+3} \|_2^2    
    + 4 \langle  \tilde{\e}_h^{n+2} - \e_h^{n+2} ,  
      \hat{\e}_h^{n+3}   \rangle   
      + 4 \| \tilde{\e}_h^{n+2} - \e_h^{n+2} \|_2^2       
\\
  \ge & 
  \| 2 \e_h^{n+2} - \e_h^{n+1} \|_2^2    
    - 8 \dtk^\frac54  \| \e_h^{n+2} \|_2^2  
	 - 8 \dtk^\frac14  \| \tilde{\e}_h^{n+2} - \e_h^{n+2} \|_2^2  
\\
  &   
    - 4 \dtk^\frac54  \| \e_h^{n+1} \|_2^2  
    - 4 \dtk^\frac14  \| \tilde{\e}_h^{n+2} - \e_h^{n+2} \|_2^2  
    + 4 \| \tilde{\e}_h^{n+2} - \e_h^{n+2} \|_2^2     
\\
  \ge & 
  \| 2 \e_h^{n+2} - \e_h^{n+1} \|_2^2    
    - \dtk (  \| \e_h^{n+2} \|_2^2  +  \| \e_h^{n+1} \|_2^2  ) 
\\
  &
    + ( 4 - 12 \dtk^\frac14 ) \| \tilde{\e}_h^{n+2} - \e_h^{n+2} \|_2^2  .    
\end{aligned} 
  \label{convergence-4-3}  
\end{equation} 
Going back~\eqref{convergence-2}, we arrive at the following estimate, for $n\leq q$: 
\begin{equation} 
	\begin{aligned}
	&  \frac{1}{4 \dtk}\big(\| \e_h^{n+2}\|_2^2-\| \e_h^{n+1}\|_2^2 
	+ \| 2 \e_h^{n+2} - \e_h^{n+1} \|_2^2 
	- \| 2 \e_h^{n+1} - \e_h^n \|_2^2	) \\ 
           & + \frac{ 5 - 13 \dtk^\frac14 }{4 \dtk} \| \tilde{\e}_h^{n+2} - \e_h^{n+2} \|_2^2           
           + \Big(\alpha_2 - \frac{\gamma_0}{12} \Big)   
		 \|\nabla_h \tilde{\e}_h^{n+2}  \|_2^2  
\\
            & - ( 1 + \delta) ( 1 + \frac{\gamma_0}{12})  \Big ( 
		    \| \nabla_h \tilde{\e}_h^{n+1} \|_2^2 
                  + \frac12 \| \nabla_h \tilde{\e}_h^n \|_2^2 \Big)   \\       
	\le \, & \mathcal{C}_\delta\sum_{p=n}^{n+2} 	 
	(\|\tilde{\e}_h^{p}\|_2^2 + \| \e_h^{p}\|_2^2 \Big) 
	+ \mathcal{C} (\dtk^4+h^4) 	. 
	\end{aligned} 
	\label{convergence-5-1} 
\end{equation}         
Meanwhile, for the terms $\|\tilde{\e}_h^p \|_2^2$, $p= n, n+1, n+2$, an application of Cauchy inequality gives 
\begin{equation} 
   \|\tilde{\e}_h^p \|_2^2 \le 2 ( \| \e_h^p\|_2^2 + \| \tilde{\e}_h^p - \e_h^p \|_2^2 ) , \quad 
   p= n, n+1, n+2 . \label{convergence-5-2} 
\end{equation}        
Its substitution into~\eqref{convergence-5-1} leads to the following inequality for $n\leq q$: 
\begin{equation} 
	\begin{aligned}
	&  \frac{1}{4 \dtk}\big(\| \e_h^{n+2}\|_2^2-\| \e_h^{n+1}\|_2^2 
	+ \| 2 \e_h^{n+2} - \e_h^{n+1} \|_2^2 
	- \| 2 \e_h^{n+1} - \e_h^n \|_2^2	) \\ 
           & + \frac{1}{4 \dtk} \| \tilde{\e}_h^{n+2} - \e_h^{n+2} \|_2^2  
               - {\mathcal C}_\delta (  \| \tilde{\e}_h^{n+1} - \e_h^{n+1} \|_2^2 
               + \| \tilde{\e}_h^{n} - \e_h^{n} \|_2^2 ) 
\\                                     
           & + \Big( \alpha_2 - \frac{\gamma_0}{12} \Big)   
		 \|\nabla_h \tilde{\e}_h^{n+2}  \|_2^2   - ( 1 + \delta) ( 1 + \frac{\gamma_0}{12})  \Big ( 
		    \| \nabla_h \tilde{\e}_h^{n+1} \|_2^2 
                  + \frac12 \| \nabla_h \tilde{\e}_h^n\|_2^2 \Big)   \\       
	\le \, & \mathcal{C}_\delta \Big( 	 
	 \| \e_h^{n+2}\|_2^2 
	+ \| \e_h^{n+1}\|_2^2 + \| \e_h^{n}\|_2^2 \Big) 
	+ \mathcal{C} (\dtk^4+h^4) 	. 
	\end{aligned} 
	\label{convergence-5-3} 
\end{equation}         

	In turn, an application of discrete Gronwall inequality \cite{Girault1986}, combined with the fact~\eqref{condition-alpha-1}, yields the desired convergence estimate at $t^{q+2}$ :
\begin{equation} 
\begin{aligned} 
  &
	\|  \e_h^n\|_2^2 
	+  \gamma_0 \dtk \sum_{p=0}^n \| \nabla_h  \tilde{\e}_h^{p} \|_2^2 \leq \mathcal{C}Te^{\mathcal{C}T} (\dtk^4+h^4), \quad \forall \, n\leq q+2\leq \left\lfloor\frac{T}{k}\right\rfloor , 
\\
  & 
	\mbox{i.e.,} \quad 
	\| \e_h^n \|_2 
	+ \Big( \gamma_0 \dtk \sum_{p=0}^n \| \nabla_h  \tilde{\e}_h^{p} \|_2^2\Big)^\frac12
	\le \mathcal{C} (\dtk^2+h^2) .
	\end{aligned} 
	\label{convergence-6-1} 
\end{equation} 
The convergence estimate~\eqref{convergence-0} has been proved at $t^{q+2}$. In addition, we see that the \textit{a-priori} assumption~\eqref{bound-2} has also been validated at the next time step $t^{q+2}$, provided that $\dtk$ and $h$ are sufficiently small. By mathematical induction, this completes the proof of Theorem~\ref{cccthm2}. $\hfill\Box$


\begin{rema} 
The condition $\alpha > 3$ is a relatively strong constraint. In fact, such a condition is used in the estimate~\eqref{I1} for $\tilde{I}_1$, since we need $\alpha > 3$ to control these Laplace terms, due to the explicit treatment of $\Delta_h \hat{{\m}}_h^{n+2}$. Meanwhile, such an inequality only stands for a theoretical difficulty, and the practical computations may not need that large value of $\alpha$. In most practical simulation examples, a value of $\alpha > 1$ would be sufficient to ensure the numerical stability of the proposed numerical scheme~\cref{scheme-1-1}-\cref{scheme-1-2}. 

In addition, the explicit treatment of the Laplace term, namely $\Delta_h \hat{\m}_h^{n+2} = 
\Delta_h ( 2 \m_h^{n+1} - \m_h^n )$, will greatly improve the numerical efficiency, since only a constant-coefficient Poisson solver is needed at each step. This crucial fact enables one to produce very robust simulation results at a much-reduced computational cost.  
\end{rema}

\begin{rema} 
In a recent work~\cite{Xie2021BDF2Linear}, a rough stability estimate for the projection step, namely, $\| \e_h \|_2 \le 2   \| \tilde{\e}_h \|_2 + {\mathcal O} (h^2)$, $\| \nabla_h \e_h \|_2 \le {\mathcal C} ( \| \e_h \|_2 + \| \nabla_h \tilde{\e}_h \|_2 ) + {\mathcal O} (h^2)$, has been proved. These inequalities are sufficient to establish the stability and convergence analysis for a semi-implicit numerical scheme, if the BDF2 temporal stencil is formulated as $\frac{1}{k} ( \frac32 \tilde{\m}_h^{n+2} - 2 \tilde{\m}_h^{n+1} + \frac12 \tilde{\m}_h^n)$. However, for the BDF2 temporal stencil formulated as $\frac{1}{k} ( \frac32 \tilde{\m}_h^{n+2} - 2 \m_h^{n+1} + \frac12 \m_h^n)$ (as given by {\bf Algorithm 2.2}), such a stability estimate are not sufficient to derive the stability and convergence analysis, due to the singular coefficient $\frac{1}{k}$. Instead, a much more refined stability estimate, as given by \eqref{lem 6-2}, \eqref{lem 6-2-2} (in Lemma~\ref{lem 6-0}), is needed to pass through the convergence analysis. The proof of these two refined inequalities has to be based on a more precise geometric analysis of the corresponding vectors, and the details will be presented in Appendix A. 

Extensive numerical experiments have demonstrated a better stability property of the temporal stencil in {\bf Algorithm 2.2} than that of {\bf Algorithm 2.1}, for physical models with large damping parameters. For the theoretical analysis of {\bf Algorithm 2.2}, the refined stability estimate \eqref{lem 6-2}, \eqref{lem 6-2-2} has played a crucial role. 
\end{rema}

\section{Conclusions}
\label{sec:conclusions}

In this paper, we have presented an optimal rate convergence analysis and error estimate for a second-order accurate, linear numerical scheme to the LL equation. The second-order backward differentiation formula is applied in the temporal discretization, the linear diffusion term is treated implicitly, while the nonlinear terms are updated by a fully explicit extrapolation formula. Afterward, a point-wise projection is applied to normalize the magnetization vector. In turn, only a linear system  independent of the updated magnetization needs to be solved at each time step, which has greatly improved the computational efficiency, and many great advantages of this numerical scheme have been reported in the numerical simulation with large damping parameters.  The error estimate has been theoretically established in the discrete $\ell^{\infty}(0,T; \ell^2)  \cap \ell^2(0,T; H_h^1)$ norm, under suitable regularity assumptions and reasonable ratio between the time step-size and the spatial mesh-size. The key difficulty of the theoretical analysis is associated with the fact that the projection step is highly nonlinear and non-convex. To overcome this subtle difficulty, we build a stability estimate for the projection step, which plays a crucial role in the derivation of the convergence analysis for the numerical scheme.


\section*{Appendix A. Proof of \cref{lem 6-0}}

\begin{proof}
	First of all, an $\| \cdot \|_\infty$ bound for the numerical error $\tilde{\e}_h$ can be derived, by the \textit{a-priori} estimate~\eqref{a priori-0}: 
	\begin{equation} 
	\|  \tilde{\e}_h \|_\infty \le  \gamma h^{-\frac12} (   \|  \tilde{\e}_h \|_2 
	+ \| \nabla_h \tilde{\e}_h \|_2 ) 
	\le   \gamma h^{-\frac12} \cdot  \dtk^{\frac{11}{8}}  
	\le \frac14  \dtk^{\frac{3}{4}} , 
	\label{a priori-0-2} 
	\end{equation}      
	provided that $\dtk$ and $h$ are sufficiently small, and under the linear refinement requirement $C_1 h \le \dtk \le C_2 h$. Notice that the inverse inequality~\eqref{inverse-1} has been applied in the first step. In turn, we observe that the following bounds are available for the numerical profile $\tilde{\m}_h$:
	\begin{align}
	&  1 -  \frac14 \dtk^{\frac34}  \le | \tilde{\m}_h | \le 1 + \frac14 \dtk^{\frac34} ,  
	\quad \mbox{at a point-wise level} ,  \label{lem 6-1-0}
	\\
	&  
	\| \nabla_h \tilde{\e}_h \|_4 \le \gamma h^{-\frac34} \| \nabla_h \tilde{\e}_h \|_2 
	\le \gamma h^{-\frac34}  \cdot  \frac12  \dtk^{\frac{11}{8}}  
	\le \dtk^{\frac12} , \label{lem 6-1-1}	
	\\	
	& 	
	\| \nabla_h \tilde{\m}_h \|_4 
	\le \| \nabla_h \um_h \|_4 	 + \| \nabla_h \tilde{\e}_h \|_4 
	\le C^* + 1 := M .   \label{lem 6-1}
	\end{align}   
	Again, the inverse inequality~\eqref{inverse-1} has been applied in the derivation. 	
	
	A careful calculation indicates that 
	\begin{equation} 
	\begin{aligned}
	{\e}_h =& \um_h - {\m}_h 
	= \um_h - \tilde{\m}_h  + \tilde{\m}_h  - \frac{\tilde{\m}_h}{|\tilde{\m}_h|}  
	= \tilde{\e}_h + \frac{\tilde{\m}_h}{|\tilde{\m}_h|}( |\tilde{\m}_h| -1 ) ,  
	\quad \mbox{so that} 
	\\
	\tilde{\e}_h = & \e_h + \tilde{\e}_{h, c} ,  \quad 
	\tilde{\e}_{h, c} := \frac{\tilde{\m}_h}{|\tilde{\m}_h|} ( 1 - |\tilde{\m}_h| ) .     
	\end{aligned} 
	\label{lem 6-3}
	\end{equation} 
	Meanwhile, at a fixed grid point $(\hat{x}_i, \hat{y}_j, \hat{z}_{\ell})\in\Omega_h^0$, we look at the triangle formed by the vectors: $\tilde{\m}_h$, $\tilde{\e}_h$ and $\um_h$. In particular, we denote the angle between $\tilde{\m}_h$ and $\um_h$ as $\theta$. Since the lengths of these three vectors have the following estimates: 
	\begin{equation*} 
	1 -  \frac14 \dtk^{\frac34}  \le | \tilde{\m}_h | \le 1 + \frac14 \dtk^{\frac34} ,  \, \, \, 
	| \tilde{\e}_h | \le \frac14  \dtk^{\frac{3}{4}}  , \, \, \, 
	| \um_h | \equiv 1 , 
	\end{equation*}    
	a careful application of Sine law indicates that 
	\begin{equation} 
	0 \le \sin \theta \le \frac14  \dtk^{\frac{3}{4}} . 
	\label{lem 6-3-2}
	\end{equation}   
	
	And also, we look at the triangle formed by the vectors: $\m_h$, $\e_h$ and $\um_h$. Two sides have equal lengths: $| \m_h | = | \um_h | =1$, and the angle between $\m_h$ and $\um_h$ is exactly $\theta$, since $\m_h = \frac{\tilde{\m}_h}{|\tilde{\m}_h|}$ is in the same direction as $\tilde{\m}_h$. In turn, the angle between vectors $\m_h$ and $\e_h$ is given by $\varphi' = \frac{\pi}{2} - \frac{\theta}{2}$. Because of this fact, the angle between vectors $\tilde{\m}_h$ and $\e_h$ is the same as $\varphi' = \frac{\pi}{2} - \frac{\theta}{2}$, 
	
	Subsequently, we denote the angle between vectors $\tilde{e}_{h, c}$ and $\e_h$ as $\varphi$. In fact, this angle has the following representation 
	\begin{equation} 
	\begin{aligned} 
	& 
	\varphi = \varphi' = \frac{\pi}{2} - \frac{\theta}{2} ,  \quad 
	\mbox{if} \, \, \, |\tilde{\m}_h| \le 1;  \quad 
	\varphi =  \pi - \varphi' = \frac{\pi}{2} + \frac{\theta}{2} ,  \quad 
	\mbox{if} \, \, \, |\tilde{\m}_h| > 1  . 
	\end{aligned}  
	\label{lem 6-3-3}
	\end{equation} 
	In either case, the following estimate is valid: 
	\begin{equation} 
	| \cos \varphi | = | \sin \frac{\theta}{2} |  \le \sin \theta \le \frac14  \dtk^{\frac{3}{4}} . 
	\label{lem 6-3-4}
	\end{equation}  
	Consequently, the definition of the point-wise inner product implies the following estimate 
	\begin{equation} 
	\begin{aligned} 
	| \tilde{\e}_h |^2 = & | \e_h + \tilde{\e}_{h, c} |^2 
	= | \e_h |^2 + | \tilde{\e}_{h, c} |^2 + 2  \e_h \cdot \tilde{\e}_{h, c} 
	\\
	= & 
	| \e_h |^2 + | \tilde{\e}_{h, c} |^2 + 2  | \e_h | \cdot | \tilde{\e}_{h, c} | \cdot \cos \varphi 
	\\
	\ge & 
	| \e_h |^2 + | \tilde{\e}_{h, c} |^2 - 2  | \e_h | \cdot | \tilde{\e}_{h, c} | 
	\cdot \frac14  \dtk^{\frac{3}{4}}  
	=  
	| \e_h |^2 + | \tilde{\e}_{h, c} |^2 - \frac12  \dtk^{\frac{3}{4}}   
	| \e_h | \cdot | \tilde{\e}_{h, c} | 
	\\
	\ge & 
	| \e_h |^2 + | \tilde{\e}_{h, c} |^2 - ( \dtk^{\frac{5}{4}}   
	| \e_h |^2 + \frac14 \dtk^{\frac14}  | \tilde{\e}_{h, c} |^2 )   
	\\
	\ge & 
	( 1 - \dtk^{\frac54} ) | \e_h |^2 + ( 1 - \dtk^{\frac14} ) | \tilde{\e}_{h} - \e_h |^2 .       
	\end{aligned}  
	\label{lem 6-3-5}
	\end{equation}   
	Notice that this estimate is valid at a point-wise level, for any fixed mesh point $(i, j, \ell)$. Therefore, a summation in space leads to the first inequality~\eqref{lem 6-2}.           	
	
	To derive the second inequality~\eqref{lem 6-2-2}, we will focus on the $D_x$ part in the discrete gradient; the analysis for the $D_y$ and $D_z$ parts can be performed in a similar manner. We begin with the following expansion: 
	\begin{equation} 
	\begin{aligned}
	D_x \e_h &= D_x {\um}_h - D_x \frac{\tilde{\m}_h}{|\tilde{\m}_h|}
	= D_x \left[ \frac{{\um}_h}{|\tilde{\m}_h|} - \frac{\tilde{\m}_h}{|\tilde{\m}_h|} \right]
	+ D_x \left[ \um_h - \frac{{\um}_h}{|\tilde{\m}_h|} \right] 
	\\
	&= D_x \frac{\tilde{\e}_h}{|\tilde{\m}_h|} - D_x \left[ \frac{{\um}_h}{|\tilde{\m}_h|}( 1 - |\tilde{\m}_h| )\right] 
	\\
	&= D_x \frac{\tilde{\e}_h}{|\tilde{\m}_h|} - D_x \left[  \frac{{\um}_h}{|\tilde{\m}_h|} 
	\frac{( \um_h + \tilde{\m}_h ) \cdot ( \um_h - \tilde{\m}_h )}{1+|\tilde{\m}_h|}  \right]   
	\\
	&= D_x \frac{\tilde{\e}_h}{|\tilde{\m}_h|} - D_x \left[  \um_h  
	\frac{( \um_h + \tilde{\m}_h ) \cdot \tilde{\e}_h }{ |\tilde{\m}_h| + |\tilde{\m}_h|^2}  
	\right]  
	\\
	&= D_x \frac{\tilde{\e}_h}{|\tilde{\m}_h|} - D_x \left[  \um_h  
	\frac{( 2 \um_h - \tilde{\e}_h ) \cdot \tilde{\e}_h }{ |\tilde{\m}_h| + |\tilde{\m}_h|^2}  
	\right]  
	\\
	&= D_x \frac{\tilde{\e}_h}{|\tilde{\m}_h|} - D_x \Big(  \um_h  
	\frac{2 \um_h  \cdot \tilde{\e}_h }{ |\tilde{\m}_h| + |\tilde{\m}_h|^2}  \Big)  
	+ D_x \Big(  \um_h  
	\frac{ | \tilde{\e}_h |^2 }{ |\tilde{\m}_h| + |\tilde{\m}_h|^2}  \Big)  ,               
	\end{aligned}	
	\label{lem 6-6}
	\end{equation} 	
	in which the identity $\tilde{\m}_h = \um_h - \tilde{\e}_h$ has been applied at the last two steps. Meanwhile, at each numerical mesh cell, from $(i, j, \ell)$ to $(i+1, j, \ell)$, the identity~\eqref{expansion-1} is valid, so that the following expansions are able to be made at the center location $(i+1/2, j, \ell)$ over the numerical cell: 
	\begin{equation} 
	\begin{aligned} 
	& 
	D_x \frac{\tilde{\e}_h}{|\tilde{\m}_h|}  - D_x \Big(  \um_h  
	\frac{2 \um_h  \cdot \tilde{\e}_h }{ |\tilde{\m}_h| + |\tilde{\m}_h|^2}  \Big)  
	= J_1 + J_2 + J_3 + J_4 + J_5 , 
	\\  
	J_1 = & A_x \Big( \frac{1}{|\tilde{\m}_h|}  \Big) D_x \tilde{\e}_h 
	-   A_x \um_h  A_x \Big( \frac{2}{ |\tilde{\m}_h| + |\tilde{\m}_h|^2}  \Big)     
	( A_x \um_h  \cdot D_x \tilde{\e}_h )  , 
	\\  
	J_2 = & - \Big( A_x^{(2)} \Big( \frac{1}{|\tilde{\m}_h|^2}  \Big) D_x |\tilde{\m}_h| \Big) 
	\tilde{\e}_h ,\quad
	J_3 =  
	- 2 ( D_x \um_h )   
	A_x \Big( \frac{ \um_h  \cdot \tilde{\e}_h }{|\tilde{\m}_h| + |\tilde{\m}_h|^2}  \Big) ,      
	\\
	J_4 = & 
	( A_x \um_h )  A_x ( \um_h  \cdot \tilde{\e}_h  )    
	A_x^{(2)} \Big( \frac{ 2 }{ ( |\tilde{\m}_h| + |\tilde{\m}_h|^2 )^2}  \Big)  
	\\
	& 
	\cdot ( D_x  |\tilde{\m}_h| + 2 A_x \tilde{\m}_h \cdot D_x \tilde{\m}_h )  ,   
	\\
	J_5 = & 
	-   A_x \um_h  A_x \Big( \frac{2}{ |\tilde{\m}_h| + |\tilde{\m}_h|^2}  \Big)     
	( D_x \um_h  \cdot A_x \tilde{\e}_h )  ,                      
	\end{aligned} 
	\label{expansion-2-1}
	\end{equation} 
	in which the nonlinear average operator $A_x^{(2)}$ is introduced as 
	\begin{equation} 
	A_x^{(2)} \Big( \frac{1}{ ( f_h )^2 } \Big)_{i,j,\ell}  
	= \frac{1}{ ( f_h )_{i,j,\ell} (f_h )_{i+1,j,\ell} } ,  \, \,  
	\mbox{for scalar grid function $f_h$} . 
	\label{expansion-2-2}
	\end{equation}   
	Moreover, by the point-wise \textit{a-priori} estimate~\eqref{lem 6-1-0} for $\tilde{\m}_h$, the following bounds are available: 
	\begin{equation} 
	\begin{aligned} 
	& 
	\frac{1}{1 + \frac14 \dtk^{\frac34} } \le A_x \Big( \frac{1}{|\tilde{\m}_h|}  \Big)  
	\le \frac{1}{1 - \frac14 \dtk^{\frac34} }   , 
	\\
	&  
	1 - \frac32 \dtk^{\frac34} \le 
	A_x \Big( \frac{2}{ |\tilde{\m}_h| + |\tilde{\m}_h|^2}  \Big) 
	\le 1 + \frac32 \dtk^{\frac34}  ,    
	\\
	& 
	\frac{1}{( 1 + \frac14 \dtk^{\frac34} )^2} \le 
	A_x^{(2)} \Big( \frac{1}{|\tilde{\m}_h|^2}  \Big)       
	\le \frac{1}{(1 - \frac14 \dtk^{\frac34})^2 } ,   
	\\
	& 
	\frac{1}{2 ( 1 + \frac14 \dtk^{\frac34} )^4} \le 
	A_x^{(2)} \Big( \frac{ 2 }{ ( |\tilde{\m}_h| + |\tilde{\m}_h|^2 )^2}  \Big)      
	\le \frac{1}{(1 - \frac14 \dtk^{\frac34})^4 }  . 
	\end{aligned} 
	\label{expansion-2-3}
	\end{equation}   
	
	For the term $J_1$, we make the following decomposition $J_1 = J_{11} + J_{12} + J_{13} $ to facilitate the analysis 
	\begin{equation} 
	\begin{aligned}  
	J_{11} =  & D_x \tilde{\e}_h 
	-   A_x \um_h  ( A_x \um_h  \cdot D_x \tilde{\e}_h )  ,   \quad  
	J_{12} =   
	\Big( A_x \Big( \frac{1}{|\tilde{\m}_h|}  \Big)- 1 \Big) D_x \tilde{\e}_h 	, 
	\\
	J_{13} = & 
	-   A_x \um_h  \Big( 
	A_x \Big( \frac{2}{ |\tilde{\m}_h| + |\tilde{\m}_h|^2}  \Big)   - 1 \Big)   
	( A_x \um_h  \cdot D_x \tilde{\e}_h )   .  
	\end{aligned} 
	\label{est-J1-1}
	\end{equation}       
	For the quantity $J_{11}$, we see that 
	\begin{equation} 
	\begin{aligned} 
	| J_{11} |^2  = & | D_x \tilde{\e}_h |^2 
	+ | A_x \um_h |^2 ( A_x \um_h  \cdot D_x \tilde{\e}_h )^2 
	-  2  ( A_x \um_h  \cdot D_x \tilde{\e}_h )^2   
	\\
	\le & 
	| D_x \tilde{\e}_h |^2 
	-  ( A_x \um_h  \cdot D_x \tilde{\e}_h )^2   ,   \quad 
	\mbox{so that} \, \, \,  | J_{11} | \le | D_x \tilde{\e}_h | ,   
	\end{aligned} 
	\label{est-J1-2}
	\end{equation}  
	in which the last step comes from the fact that $| A_x \um_h |^2  \le 1$. For the quantity $J_{12}$, the \textit{a-priori} bounds~\eqref{expansion-2-3} imply that 
	\begin{equation} 
	J_{12} =   
	\Big| A_x \Big( \frac{1}{|\tilde{\m}_h|}  \Big)- 1 \Big| \cdot | D_x \tilde{\e}_h | 
	\le \frac12 \dtk^{\frac34}  | D_x \tilde{\e}_h |  .  \label{est-J1-3}
	\end{equation}  	
	Similarly, the quantity $J_{13}$ can be analyzed as 
	\begin{equation} 
	\begin{aligned} 
	J_{13} = & 
	| A_x \um_h | \cdot \Big| 
	A_x \Big( \frac{2}{ |\tilde{\m}_h| + |\tilde{\m}_h|^2}  \Big)   - 1 \Big|    
	\cdot |  A_x \um_h | \cdot | D_x \tilde{\e}_h |    
	\\
	\le & 
	( 1 + C h^2 )^2 ( \frac{1}{(1 - \frac14 \dtk^{\frac34})^2 } | D_x \tilde{\e}_h |   
	\le ( 1 + \frac34 \dtk^{\frac34} ) | D_x \tilde{\e}_h |  .      
	\end{aligned} 
	\label{est-J1-4}
	\end{equation}       
	Consequently, a combination of~\eqref{est-J1-2}-\eqref{est-J1-4} leads to 
	\begin{equation} 
	| J_1 | \le | J_{11} | + | J_{12} | + | J_{13} | 
	\le ( 1 + \frac54 \dtk^{\frac34} ) | D_x \tilde{\e}_h | 	,  
	\label{est-J1-5}
	\end{equation} 
	which turns out to be a point-wise inequality. In turn, a summation in space implies that 
	\begin{equation} 
	\| J_1 \|_2 \le ( 1 + \frac54 \dtk^{\frac34} ) \| D_x \tilde{\e}_h \|_2 .   
	\label{est-J1-6}
	\end{equation} 
	
	For the term $J_2$, we make use of the \textit{a-priori} bounds~\eqref{expansion-2-3}, as well as the fact that $| D_x | \tilde{\m}_h | | \le | D_x \tilde{\m}_h |$, so that an application of discrete H\"older inequality gives 
	\begin{equation} 
	\begin{aligned} 
	\| J_2 \|_2 \le &  \Big\| A_x^{(2)} \Big( \frac{1}{|\tilde{\m}_h|^2} \Big) \Big\|_\infty 
	\cdot \| D_x \tilde{\m}_h \|_4 \cdot \| \tilde{\e}_h  \|_4 
	\le ( 1 + \frac34 \dtk^{\frac34} ) M   \| \tilde{\e}_h  \|_4 , 
	\end{aligned} 
	\label{est-J2}
	\end{equation}     
	in which the \textit{a-priori} $W_h^{1,4}$ bound~\eqref{lem 6-1} for the numerical profile $\tilde{\m}_h$ has been applied. The other terms in the expansion~\eqref{expansion-2-1} can be analyzed in a similar manner.   
	\begin{equation} 
	\begin{aligned} 
	\| J_3 \|_2 \le & 
	\| D_x \um_h \|_\infty    
	\cdot \max \Big( \frac{ 2 }{|\tilde{\m}_h| + |\tilde{\m}_h|^2}  \Big) 
	\cdot \| \um_h \|_\infty \cdot \| \tilde{\e}_h  \|_2     
	\\
	\le & 
	C^* ( 1 + \frac34 \dtk^{\frac34} ) \cdot 1 \cdot  \| \tilde{\e}_h  \|_2   
	\le   ( 1 + \frac34 \dtk^{\frac34} ) M  \| \tilde{\e}_h  \|_2 , 
	\\
	\| J_4 \|_2 \le & 
	\| \um_h \|_\infty^2  \cdot \| \tilde{\e}_h \|_4     
	\cdot \max \Big( \frac{ 2 }{ ( |\tilde{\m}_h| + |\tilde{\m}_h|^2 )^2}  \Big)  
	\\
	& 
	\cdot ( 1 + 2 \| \tilde{\m}_h \|_\infty ) \| D_x \tilde{\m}_h \|_4   
	\le \frac32 ( 1 + \frac32 \dtk^{\frac34} )  M \| \tilde{\e}_h \|_4 ,        
	\\
	\| J_5 \|_2 \le & 
	\| \um_h  \|_\infty  \cdot \max \Big( \frac{2}{ |\tilde{\m}_h| + |\tilde{\m}_h|^2}  \Big)     
	\cdot \| D_x \um_h  \|_\infty \cdot \| \tilde{\e}_h \|_2 
	\\
	\le & 
	( 1 + \frac32 \dtk^{\frac34} ) C^* \| \tilde{\e}_h \|_2  
	\le ( 1 + \frac32 \dtk^{\frac34} ) M \| \tilde{\e}_h \|_2  .                       
	\end{aligned} 
	\label{est-J-more}  
	\end{equation}  
	Therefore, a substitution of~\eqref{est-J1-6}, \eqref{est-J2} and \eqref{est-J-more} into~\eqref{expansion-2-1} leads to 
	\begin{equation} 
	\begin{aligned} 
	& 
	\Big\| D_x \frac{\tilde{\e}_h}{|\tilde{\m}_h|}  - D_x \Big(  \um_h  
	\frac{2 \um_h  \cdot \tilde{\e}_h }{ |\tilde{\m}_h| + |\tilde{\m}_h|^2}  \Big)   \Big\|_2 
	\\
	\le & 
	( 1 + \frac54 \dtk^{\frac34} ) \| D_x \tilde{\e}_h \|_2 
	+ 3M ( \| \tilde{\e}_h \|_2 + \| \tilde{\e}_h \|_4 ) . 
	\end{aligned} 
	\label{est-gradient-2} 
	\end{equation}   
	
	The analysis for the last term on the right hand side of~\eqref{lem 6-6} can be similarly carried out; some technical details are skipped for the sake of brevity. 
	\begin{equation} 
	\begin{aligned} 
	& 
	\Big\| D_x \Big(  \um_h  
	\frac{ | \tilde{\e}_h |^2 }{ |\tilde{\m}_h| + |\tilde{\m}_h|^2}  \Big)  \Big\|_2   
	\\
	\le & 
	{\mathcal C} \Big(  \|  \um_h  \|_\infty \cdot \max 
	\Big( \frac{2 }{ |\tilde{\m}_h| + |\tilde{\m}_h|^2}  \Big)  
	\cdot \| \tilde{\e}_h \|_\infty  \cdot \| D_x \tilde{\e}_h \|_2   
	\\
	& 
	+  \|  D_x \um_h  \|_\infty \cdot \max 
	\Big( \frac{1}{ |\tilde{\m}_h| + |\tilde{\m}_h|^2}  \Big)  
	\cdot \| \tilde{\e}_h \|_\infty  \cdot \| \tilde{\e}_h \|_2     
	\\
	& 
	+ \|  \um_h  \|_\infty \cdot \max 
	\Big( \frac{1}{ ( |\tilde{\m}_h| + |\tilde{\m}_h|^2 )^2 }  \Big)  
	\cdot ( 1 + 2 \| \tilde{\m}_h \|_\infty ) \| D_x \tilde{\m}_h \|_4    
	\\
	&     
	\cdot \| \tilde{\e}_h \|_\infty  \cdot \| \tilde{\e}_h \|_4  \Big)                       
	\\
	\le & 
	{\mathcal C}  \dtk^{\frac34} ( \| D_x \tilde{\e}_h \|_2  
	+  \| \tilde{\e}_h \|_2 + \| \tilde{\e}_h \|_4 ) 
	\le   \dtk^{\frac58} ( \| D_x \tilde{\e}_h \|_2  
	+  \| \tilde{\e}_h \|_2 + \| \tilde{\e}_h \|_4 )  .     	
	\end{aligned} 
	\label{est-gradient-3}
	\end{equation} 	
	
	Finally, a substitution of~\eqref{est-gradient-2} and \eqref{est-gradient-3} into \eqref{lem 6-6} yields 
	\begin{equation} 
	\begin{aligned} 
	\| D_x \e_h \|_2 \le  & 
	( 1 + 2 \dtk^{\frac58} ) \| D_x \tilde{\e}_h \|_2 
	+ ( 3M + 1) ( \| \tilde{\e}_h \|_2 + \| \tilde{\e}_h \|_4 )  ,  \quad \mbox{so that} 
	\\
	\| D_x \e_h \|_2^2 \le  & 
	( 1 + 5 \dtk^{\frac58} ) \| D_x \tilde{\e}_h \|_2^2 
	+ ( 3M + 1)^2 ( \| \tilde{\e}_h \|_2 + \| \tilde{\e}_h \|_4 )^2 
	\\
	& 
	+ 4 ( 3M+1) \| D_x \tilde{\e}_h \|_2 ( \| \tilde{\e}_h \|_2 + \| \tilde{\e}_h \|_4 )    
	\\
	\le  & 
	( 1 + 5 \dtk^{\frac58} ) \| D_x \tilde{\e}_h \|_2^2 
	+ 2 ( 3M + 1)^2 ( \| \tilde{\e}_h \|_2^2 + \| \tilde{\e}_h \|_4^2 ) 
	\\
	& 
	+ \frac{\delta}{4} \| D_x \tilde{\e}_h \|_2^2 
	+ 16 (3 M +1)^2 \delta^{-1} ( \| \tilde{\e}_h \|_2 + \| \tilde{\e}_h \|_4 )^2  
	\\
	\le  & 
	( 1 + \frac{\delta}{2} ) \| D_x \tilde{\e}_h \|_2^2 
	+ ( 32 \delta^{-1}  + 2) ( 3M + 1)^2 ( \| \tilde{\e}_h \|_2^2 + \| \tilde{\e}_h \|_4^2 )   ,  
	\end{aligned}                	
	\label{est-gradient-4-1}
	\end{equation} 	
	for any $\delta > 0$, provided that $\dtk$ is sufficiently small. Similar estimates can be derived for the gradient in the $y$ and $z$ directions; the technical details are skipped for the sake of brevity. 
	\begin{equation} 
	\begin{aligned}
	\| D_y \e_h \|_2^2 \le & 
	( 1 + \frac{\delta}{2} ) \| D_y \tilde{\e}_h \|_2^2 
	+ ( 32 \delta^{-1}  + 2) ( 3M + 1)^2 ( \| \tilde{\e}_h \|_2^2 + \| \tilde{\e}_h \|_4^2 )   , 
	\\
	\| D_z \e_h \|_2^2 \le & 
	( 1 + \frac{\delta}{2} ) \| D_z \tilde{\e}_h \|_2^2 
	+ ( 32 \delta^{-1}  + 2) ( 3M + 1)^2 ( \| \tilde{\e}_h \|_2^2 + \| \tilde{\e}_h \|_4^2 )   .                       
	\end{aligned}	
	\label{est-gradient-4-2} 
	\end{equation} 	
	Then we arrive at 
	\begin{equation} 
	\| \nabla_h \e_h \|_2^2 \le  
	( 1 + \frac{\delta}{2} ) \| \nabla_h \tilde{\e}_h \|_2^2 
	+ 3 ( 32 \delta^{-1}  + 2) ( 3M + 1)^2 ( \| \tilde{\e}_h \|_2^2 + \| \tilde{\e}_h \|_4^2 ) .              	
	\label{est-gradient-4-3}
	\end{equation} 	     
	Meanwhile, by the discrete Sobolev inequality~\eqref{Sobolev-1} (in Lemma~\ref{lem: Sobolev-1}), we get 
	\begin{align} 
	\| \tilde{\e}_h \|_4 
	\le  \mathcal{C}  ( \| \tilde{\e}_h \|_2  	
	+   	\| \tilde{\e}_h \|_2^\frac14 \cdot \| \nabla_h \tilde{\e}_h \|_2^\frac34 ) , 
	\label{est-gradient-4-4} 
	\end{align}
	so that 
	\begin{equation}  
	\begin{aligned} 
	& 
	3 ( 32 \delta^{-1}  + 2) ( 3M + 1)^2 \| \tilde{\e}_h \|_4^2 
	\\
	\le & 3 ( 32 \delta^{-1}  + 2) ( 3M + 1)^2 \mathcal{C}  ( \| \tilde{\e}_h \|_2^2  	
	+   	\| \tilde{\e}_h \|_2^\frac12 \cdot \| \nabla_h \tilde{\e}_h \|_2^\frac32 ) 
	\\
	\le & 
	{\mathcal C}_\delta \| \tilde{\e}_h \|_2^2   
	+ \frac{\delta}{2}  \| \nabla_h \tilde{\e}_h \|_2^2  ,   \quad 
	\forall \delta > 0 ,         
	\end{aligned} 
	\label{est-gradient-4-5} 
	\end{equation}
	in which the Young's inequality has been applied in the last step. Going back~\eqref{est-gradient-4-3}, we obtain 
	\begin{equation} 
	\begin{aligned} 
	\| \nabla_h \e_h \|_2 \le  & 
	( 1 + \frac{\delta}{2} ) \| \nabla_h \tilde{\e}_h \|_2 
	+ \frac{\delta}{2}  \| \nabla_h \tilde{\e}_h \|_2 
	+ {\mathcal C}_\delta \| \tilde{\e}_h \|_2  
	\le  
	( 1 + \delta ) \| \nabla_h \tilde{\e}_h \|_2 
	+ {\mathcal C}_\delta \| \tilde{\e}_h \|_2 ,     
	\end{aligned}           	
	\label{est-gradient-4-6}
	\end{equation} 	     
	provided that $k$ is sufficiently small. This finishes the proof of Lemma~\ref{lem 6-0}. 	
\end{proof}

\section*{Appendix B. Proof of \cref{lem 7-0}}

\begin{proof}
Here $\um_h^{(1)}$ and  $\um_h^{(2)}$ serve as the exact solution at different time steps and $\tilde{\m}_h^{(1)},\tilde{\m}_h^{(2)}\in X$ are the corresponding numerical solutions.

     The $\| \cdot \|_\infty$ bounds for the error functions $\tilde{\e}_h^{(q)}$ can be derived, with the help of the \textit{a-priori} estimate~\eqref{a priori-n-1}: 
\begin{equation} 
\begin{aligned} 
  & 
    \|  \tilde{\e}_h^{(q)} \|_\infty \le  \gamma h^{-\frac12} (   \|  \tilde{\e}_h^{(q)} \|_2 
    + \| \nabla_h \tilde{\e}_h^{(q)} \|_2 ) 
    \le   \gamma h^{-\frac12} \cdot   \dtk^{\frac{11}{8}}  
    \le \frac14  \dtk^{\frac{3}{4}}  , 
\\
  & 
  \mbox{so that} \, \,  
  1 -  \frac14 \dtk^{\frac34}  \le | \tilde{\m}_h^{(q)} | \le 1 + \frac14 \dtk^{\frac34} ,  
	\, \,  \mbox{at a point-wise level, \, $q=1, 2$}. 
\end{aligned} 
     \label{a priori-n-3} 
\end{equation}      
	    
For the numerical error functions at different time steps, the decomposition~\eqref{lem 6-3} is still valid: 
\begin{equation} 
      \tilde{\e}_h^{(q)} = \e_h^{(q)} + \tilde{\e}_{h, c}^{(q)} ,  \quad 
        \tilde{\e}_{h, c}^{(q)} := \frac{\tilde{\m}_h^{(q)}}{|\tilde{\m}_h^{(q)} |}  
        ( 1 - | \tilde{\m}_h^{(q)} | ) .  \label{lem 7-2}
\end{equation} 
At a fixed grid point $(\hat{x}_i, \hat{y}_j, \hat{z}_k)$, we look at the triangle formed by the vectors: $\tilde{\m}_h^{(q)}$, $\tilde{\e}_h^{(q)}$, $\um_h^{(q)}$, and denote the angle between $\tilde{\m}_h^{(q))}$ and $\um_h^{(q)}$ as $\theta_q$. In turn, the estimate~\eqref{lem 6-3-2} is laid for each $\theta_q$: 
\begin{equation}  
   \begin{aligned} 
     & 
   1 -  \frac14 \dtk^{\frac34}  \le | \tilde{\m}_h^{(q)} | \le 1 + \frac14 \dtk^{\frac34} ,  \, \, \, 
    | \tilde{\e}_h | \le \frac14  \dtk^{\frac{3}{4}}  , \, \, \, 
    | \um_h^{(q)} | \equiv 1 , 
\\
  & 
  0 \le \sin \theta_j \le \frac14  \dtk^{\frac{3}{4}} , \quad q= 1, 2 . 
\end{aligned} 
  \label{lem 7-3-1}
\end{equation}   
Similarly, in the triangle formed by the vectors: $\m_h^{(q)}$, $\e_h^{(q)}$, $\um_h^{(q)}$, two sides have equal lengths: $| \m_h^{(q)} | = | \um_h^{(q)} | =1$, and the angle between $\m_h^{(q)}$ and $\um_h^{(q)}$ is exactly $\theta_q$. In turn, the angle between vectors $\m_h^{(q)}$ and $\e_h^{(q)}$, as well as the angle between vectors $\um_h^{(q)}$ and $\e_h^{(q)}$ is given by $\varphi^{(q)} = \frac{\pi}{2} - \frac{\theta_q}{2}$. 

Meanwhile, we denote the angle between $\um_h^{(1)}$ and $\um_h^{(2)}$ as $\theta^*$. By the \textit{a-priori} assumption and the fact that $| \um_h^{(1)} | = | \um_h^{(2)} | =1$, we have an estimate for $\theta^*$: 
\begin{equation} 
  2 \sin \frac{\theta^*}{2} \le \frac14 \dtk^{\frac78} . 
  \label{est-theta-1} 
\end{equation} 
Subsequently, we denote the angle between $\tilde{\m}_h^{(1)}$ and $\e_h^{(2)}$ as $\phi^*$. By the above analyses, we see that 
\begin{equation} 
  \phi^* = \varphi^{(2)} \pm ( \theta_1 + \theta^* )  
  = \frac{\pi}{2} - \frac{\theta_2}{2} \pm ( \theta_1 + \theta^* )  .  
    \label{est-phi-1} 
\end{equation}   
In either case, $\phi^* = \frac{\pi}{2} - \frac{\theta_2}{2} + ( \theta_1 + \theta^* )$ or $\phi^* = \frac{\pi}{2} - \frac{\theta_2}{2} - ( \theta_1 + \theta^* )$, the following estimate is valid: 
\begin{equation} 
\begin{aligned} 
  | \cos \phi^* | = & | \sin ( \frac{\theta_2}{2} \pm ( \theta_1 + \theta^* ) ) |  
  \le \sin ( \frac{\theta_2}{2} + \theta_1 + \theta^* )  
\\
  \le & 
   \sin \frac{\theta_2}{2} + \sin \theta_1 + \sin \theta^*  
   \le   \sin \theta_2 + \sin \theta_1 + \sin \theta^*  
\\
  \le &   
  \frac14  \dtk^{\frac{3}{4}}  + \frac14  \dtk^{\frac{3}{4}} + \frac14 \dtk^{\frac78}  
  \le \frac34 \dtk^{\frac34} ,  
\end{aligned} 
   \label{lem 7-4-1}
\end{equation}  
in which we have used the triangular inequality, $\sin (s_1 + s_2 + s_3) \le \sin s_1 + \sin s_2 + \sin s_3$, provided that $s_1, s_2,s_3>0$ are sufficiently small. As a consequence, the definition of the point-wise inner product implies the following estimate 
\begin{equation} 
\begin{aligned} 
   \Big| ( \tilde{\e}_h^{(1)} - \e_h^{(1)} ) \cdot  \e_h^{(2)} \Big| 
   = &  | \tilde{\e}_h^{(1)} - \e_h^{(1)}  |  \cdot | \tilde{\e}_h^{(2)} | 
   \cdot | \cos \phi^* |  
\\
  \le & 
   | \tilde{\e}_h^{(1)} - \e_h^{(1)}  |  \cdot | \tilde{\e}_h^{(2)} | 
   \cdot   \frac34 \dtk^{\frac34}   
\\
  \le & 
  \dtk^{\frac14} | \tilde{\e}_h^{(1)} - \e_h^{(1)}  |^2  
  + \dtk^{\frac45} | \tilde{\e}_h^{(2)} |^2 .        
\end{aligned}  
  \label{lem 7-4-2}
\end{equation}   
Again, this estimate is valid at a point-wise level for $\x\in\Omega_h^0$. Therefore, a summation in space leads to the first inequality~\eqref{lem 7-1}.           	
\end{proof} 
\section*{Acknowledgments}
This work is supported in part by the grants NSFC 11971021 (J.~Chen), NSF DMS-2012669 (C.~Wang), NSFC 11771036 and 12171041 (Y.~Cai).

\bibliographystyle{amsplain}
\bibliography{references}
%
%
%
%

\end{document}